\documentclass[a4paper,10pt]{amsart}

\usepackage{amsmath}
\usepackage{amsthm}
\usepackage{amssymb}
\usepackage{hyperref}

\usepackage{mathrsfs}
\usepackage[all]{xy}
\usepackage{color}

\numberwithin{equation}{section}

\newcommand{\SDR}[5]{\xymatrix{*[r]{#1} \ar@<1ex>[r]^-{#3} \ar@(ul,dl)[]_{#5} & #2 \ar@<1ex>[l]^-{#4}}}
\newcommand{\bigSDR}[5]{\xymatrix{*[r]{#1} \ar@<1ex>[rr]^-{#3} \ar@(ul,dl)[]_{#5} && #2 \ar@<1ex>[ll]^-{#4}}}
\newcommand{\bigbigSDR}[5]{\xymatrix{*[r]{#1} \ar@<1ex>[rrr]^-{#3} \ar@(ul,dl)[]_{#5} &&& #2 \ar@<1ex>[lll]^-{#4}}}

\newcommand{\adjunction}[4]{\xymatrix{ {#1} \ar@<1ex>[r]^-{#3} & {#2} \ar@<1ex>[l]^-{#4}}}
\newcommand{\bigadjunction}[4]{\xymatrix{ #1 \ar@<1ex>[rr]^-{#3} && #2 \ar@<1ex>[ll]^-{#4}}}
\newcommand{\hugeadjunction}[4]{\xymatrix{ #1 \ar@<1ex>[rrr]^-{#3} &&& #2 \ar@<1ex>[lll]^-{#4}}}
\newcommand{\varadjunction}[4]{\xymatrix{ #3 \colon #1 \ar@<1ex>[r] & #2 \colon #4 \ar@<1ex>[l]}}

\newcommand{\tensor}{\otimes}
\newcommand{\smsh}{\wedge}
\newcommand{\cotensor}{\square}
\newcommand{\Hom}{\operatorname{Hom}}
\newcommand{\Map}{\operatorname{Map}}

\newcommand{\ZZ}{\mathbb{Z}}

\newcommand{\as}{\text{!`}}

\newcommand{\Ho}{\operatorname{Ho}}

\newcommand{\Ch}{\mathsf{Ch}}
\newcommand{\Sp}{\mathsf{Sp}}

\newcommand{\Coring}{\mathsf{Coring}}
\newcommand{\CORING}{\mathsf{CORING}}

\newcommand{\VCAT}{\mathsf{CAT}_\VV}

\newcommand{\Alg}{\mathsf{Alg}}
\newcommand{\ALG}{\mathsf{ALG}}
\newcommand{\Coalg}{\mathsf{Coalg}}

\newcommand{\Ab}{\mathscr{A}b}

\newcommand{\kk}{\Bbbk}

\newcommand{\CC}{\mathscr{C}}
\newcommand{\DD}{\mathscr{D}}

\newcommand{\VV}{\mathscr{V}}
\newcommand{\UU}{\mathcal{U}}

\newcommand{\coequalizer}[5]{\xymatrix{{#1} \ar@<0.5ex>[r]^-{#4} \ar@<-0.5ex>[r]_-{#5} & {#2} \ar[r] & {#3}}}
\newcommand{\equalizer}[5]{\xymatrix{{#1} \ar[r] & {#2} \ar@<0.5ex>[r]^-{#4} \ar@<-0.5ex>[r]_-{#5} & {#3}}}

\newcommand{\Desc}{\operatorname{Desc}}

\newcommand{\Can}{\operatorname{Can}}
\newcommand{\Prim}{\operatorname{Prim}}

\newcommand{\vp}{\varphi}
\newcommand{\ve}{\varepsilon}
\newcommand{\Om}{\Omega}

\newcommand\si{s^{-1}}

\newtheorem{theorem}{Theorem}[section]
\newtheorem{proposition}[theorem]{Proposition}

\newtheorem{corollary}[theorem]{Corollary}
\newtheorem{lemma}[theorem]{Lemma}

\theoremstyle{definition}
\newtheorem{definition}[theorem]{Definition}

\newtheorem{example}[theorem]{Example}
\newtheorem{remark}[theorem]{Remark}
\newtheorem{axiom}[theorem]{Axiom}

\newtheorem{hypothesis}[theorem]{Hypothesis}
\newtheorem{notation}[theorem]{Notation}

\numberwithin{theorem}{section}
\title{Homotopical Morita theory for corings}
\author{Alexander Berglund and Kathryn Hess}
\address{Department of Mathematics, Stockholm University, SE-106 91 Stockholm, Sweden}
\email{alexb@math.su.se}
\address{MATHGEOM, \'Ecole Polytechnique F\'ed\'erale de Lausanne, CH-1015 Lausanne, Switzerland}
\email{kathryn.hess@epfl.ch}

\thanks{This material is based upon work partially supported by the National Science Foundation under Grant No.~0932078000 while the second author was in residence at the Mathematical Sciences Research Institute in Berkeley, California, during the Spring 2014 semester. The second author was also supported during this project by the Swiss National Science Foundation, Grant No. 200020\_144393.}

\begin{document}

\begin{abstract}
A coring $(A,C)$ consists of an algebra $A$ in a symmetric monoidal category and a coalgebra $C$ in the monoidal category of $A$-bimodules. Corings and their comodules arise naturally in the study of Hopf-Galois extensions and descent theory, as well as in the study of Hopf algebroids. In this paper, we address the question of when two corings $(A,C)$ and $(B,D)$ in a symmetric monoidal model category $\VV$ are homotopically Morita equivalent, i.e., when their respective categories of comodules $\VV_A^C$ and $\VV_B^D$ are Quillen equivalent.

The category of comodules over the trivial coring $(A,A)$ is isomorphic to the category $\VV_A$ of $A$-modules, so the question englobes that of when two algebras are homotopically Morita equivalent. We discuss this special case in the first part of the paper, extending previously known results.

To approach the general question, we introduce the notion of a \emph{braided bimodule} and show that adjunctions between $\VV_A$ and $\VV_B$ that lift to adjunctions between $\VV_A^C$ and $\VV_B^D$ correspond precisely to braided bimodules. We then give criteria, in terms of homotopic descent, for when a braided bimodule induces a Quillen equivalence between $\VV_A^C$ and $\VV_B^D$. In particular, we obtain criteria for when a morphism of corings induces a Quillen equivalence, providing a homotopic generalization of results by Hovey and Strickland on Morita equivalences of Hopf algebroids. As an illustration of the general theory, we examine homotopical Morita theory for corings in the category of chain complexes over a commutative ring.
\end{abstract}

\maketitle

\tableofcontents

\section{Introduction}
The study of equivalences between categories of comodules over coalgebras over a field was initiated by Takeuchi \cite{takeuchi} and is commonly referred to as \emph{Morita-Takeuchi theory}. The more general question of when categories of comodules over corings, $\VV_A^C$ and $\VV_B^D$, are strictly equivalent categories has also been studied, when $A$ and $B$ are are algebras over a commutative ring; see \cite{brzezinski-wisbauer} and the references therein.

In this paper, we address the homotopical, or derived, version of this question. Suppose that we have a notion of weak equivalence in $\VV_A^C$ such that the homotopy category $\Ho(\VV_A^C)$ (or derived category) may be formed.
\begin{quote}
\emph{When are $\VV_A^C$ and $\VV_B^D$ derived equivalent?}
\end{quote}
This question arose out of attempts to understand relations among homotopic Hopf-Galois extensions, Grothendieck descent, and Koszul duality. As it turns out, each of these notions can be expressed as asserting a derived equivalence between particular corings.

{\it Grothendieck descent.}
A morphism of rings $\varphi\colon R\to S$ satisfies effective homotopic descent if the canonical functor $\Can_\varphi\colon \VV_S \to \VV_S^{S\tensor_R S}$ is a derived equivalence, where $(S,S\tensor_R S)$ is the \emph{descent coring} associated to $\varphi$ (see \cite{hess:descent}).

{\it Koszul duality.}
For a quadratic algebra $A$ over a field $\Bbbk$, with quadratic dual coalgebra $A^{\as}$, the Koszul complex furnishes an adjunction between the categories $\VV_A^A$ and $\VV_\Bbbk^{A^{\as}}$. The algebra $A$ is a Koszul algebra \cite{priddy} if and only if this adjunction is a derived equivalence.

{\it Hopf-Galois extensions of ring spectra.}
A morphism $A\to B$ of commutative ring spectra, with $A$ weakly equivalent to the homotopy $H$-coinvariants for a coaction of a commutative Hopf algebra spectrum $H$ on $B$, is an $H$-Hopf-Galois extension if the change of corings adjunction $\VV_B^{B\smsh_A B} \to \VV_B^{B\smsh H}$, induced by the Galois map $h\colon B\smsh_A B \to B\smsh H$, is a derived equivalence, see \cite{rognes,hess:hhg}.

%


To accommodate these and other examples, we need a quite general setup. We employ Quillen's homotopical algebra and work with corings in a closed symmetric monoidal model category $\VV$ (see Appendix \ref{app:enriched}). In the examples above, $\VV$ is, respectively, the category $\Ch_\kk$ of chain complexes of $\kk$-vector spaces, and a suitable symmetric monoidal category of spectra.

It is a delicate problem to prove the existence of a model structure on $\VV_A^C$ with weak equivalences the underlying weak equivalences in $\VV$. At the moment, the most general results we know come from \cite{hkrs}, which establishes the existence of such a model structure for every coring $(A,C)$ for the following examples of base categories $\VV$.
\begin{itemize}
\item The category $\Ch_R$ of chain complexes over a commutative ring $R$, with weak equivalences either the quasi-isomorphisms or the chain homotopy equivalences.
\item The category $\Sp^\Sigma$ of symmetric spectra with stable equivalences as weak equivalences.
\end{itemize}


Let us also point out here that every symmetric monoidal category $\VV$ admits a model structure where the weak equivalences are the isomorphisms, and where every morphism is a fibration and a cofibration. This implies that every homotopical result proved in this paper specializes to its strict analog.

\subsection*{Results}
In this paper, we consider adjunctions between $\VV_A^C$ and $\VV_B^D$ that lift a given adjunction between the underlying module categories $\VV_A$ and $\VV_B$. We introduce the notion of a \emph{braided bimodule} and prove that every adjunction of the above form is governed by a braided bimodule (see Definition \ref{def:braided bimodule} and Theorem \ref{thm:relative classification}).

Next, we turn to the question of when a braided bimodule $X$ gives rise to a Quillen equivalence. Assuming the underlying $B$-module $X_B$ is dualizable, we show that there is a canonically associated $B$-coring $X_*(C)$, which we call the \emph{canonical coring}, and a morphism of $B$-corings $g_T\colon X_*(C)\to D$ (Proposition \ref{prop:universal coring}). The following special case of our main result (Theorem \ref{thm:bb}) characterizes when a braided bimodule gives rise to a Quillen equivalence in terms of effective homotopic descent and the morphism of corings $g_{T}$.

\begin{theorem}
Let $(A,C)$ be a flat coring in $\VV$, and let $(X,T)\colon (A,C)\to (B,D)$ be a braided bimodule. Suppose that $X_B$ is strictly dualizable, that $-\tensor_A X$ is a left Quillen functor, and that ${}_AX$ is strongly homotopy flat. The induced functor $T_*\colon \VV_A^C\to \VV_B^D$ is a Quillen equivalence if and only if
\begin{enumerate}
\item $X$ satisfies effective homotopic descent with respect to $C$, and
\item the morphism of $B$-corings $g_T\colon X_*(C) \to D$ is a copure weak equivalence.
\end{enumerate}
\end{theorem}

See Definitions \ref{defn:copure} and \ref{defn:homotopic descent} for the notions of effective homotopic descent and copure weak equivalence.

The question arises of whether one can find reasonable criteria for effective homotopic descent and copure weak equivalences. The following result (appearing as Theorem \ref{thm:homotopic descent} in the main text) is a vast generalization of Grothendieck's classical theorem on faithfully flat descent for homomorphisms of commutative rings.

\begin{theorem}[Homotopic faithfully flat descent]
If ${}_AX_B$ is a bimodule that is strictly dualizable and cofibrant as a right $B$-module and strongly homotopy flat as a left $A$-module, then the following are equivalent.
\begin{enumerate}
\item $X$ is homotopy faithful as a left $A$-module.
\item $X$ satisfies homotopic descent with respect to every flat coring $(A,C)$.
\end{enumerate}
\end{theorem}
The classical theorem is recovered by specializing to the trivial model structure on the category of abelian groups (where the weak equivalences are the isomorphisms), taking $C$ to be the trivial coring, and letting $X$ be the bimodule $B$.

Adjunctions governed by braided bimodules are not the most general kind of adjunctions between categories of comodules over corings, but they are enough for the applications to homotopic descent and homotopic Hopf-Galois extensions that motivated this work, cf.~\cite{berglund-hess:hhg}.

The paper is structured as follows. In Section \ref{sec:morita}, we discuss homotopical Morita theory for algebras in $\VV$. In Section \ref{sec:corings}, we introduce and study $\VV$-categories of comodules over corings. Here we introduce the two main new concepts: \emph{braided bimodules} (Section \ref{subsec:braided bimodules}) and the \emph{canonical coring} associated to a braided bimodule whose underlying bimodule is strictly dualizable (Section \ref{subsec:canonical coring}). Section \ref{sec:m-t} contains our main results on homotopical Morita theory for corings. In Section \ref{sec:examples}, we apply the general theory to the case when $\VV$ is the category of chain complexes over a commutative ring. In the appendices we recall necessary elements of the theory of enriched model categories and the theory of dualizable objects in a symmetric monoidal category.

\section{Homotopical Morita theory for algebras} \label{sec:morita}
Classical Morita theory provides criteria for equivalences of categories of modules over rings. In this section we answer the corresponding question in the homotopical setting: for $\VV$ a symmetric monoidal model category and $A$ and $B$ algebras (monoids) in $\VV$, when are the $\VV$-model categories $\VV_A$ and $\VV_B$ Quillen equivalent? Homotopical Morita theory for unbounded differential graded algebras was studied by Dugger-Shipley \cite{dugger-shipley}, and for ring spectra by Schwede and Shipley \cite{schwede-shipley2} and Shipley \cite{shipley}. For derived categories, see Rickard \cite{rickard:morita,rickard:derived}. We give a self-contained and relatively short account that subsumes known results in these settings. Our results also apply to some new cases, such as the unstable model categories of non-negatively graded differential graded algebras and topological spaces.

We begin this section by recalling and elaborating somewhat on the homotopy theory of modules in a monoidal model category.  
We completely characterize enriched adjunctions between module categories, then use this characterization to provide conditions under which such an adjunction is a Quillen adjunction.  Finally, we prove a homotopical version of the usual Morita theorem, giving criteria under which an adjunction between module categories is a Quillen equivalence.

\subsection{$\VV$-model categories of modules}
Let $(\VV,\tensor,\kk)$ be a monoidal category. An \emph{algebra} (also known as a \emph{monoid}) in $\VV$ is an object $A$ together with two maps $\mu\colon A\tensor A\rightarrow A$ and $\eta\colon \kk\rightarrow A$ that satisfy the usual associativity and unit axioms. We let $\Alg_{\VV}$ denote the category of algebras in $\VV$. Dually, the category of \emph{coalgebras} in $\VV$, which are endowed with a a coassociative comultiplication and counit, is denoted $\Coalg_{\VV}$.

A right module over $A$ is an object $M$ in $\VV$ together with a map $\rho\colon M\tensor A\rightarrow M$ satisfying the usual axioms for a right action. We let $\VV_A$ denote the category of right $A$-modules in $\VV$.  {We usually omit the multiplication and unit from the notation for an algebra and the action map from the notation for an $A$-module.}

\begin{proposition}\label{prop:Amod-Vstruct}  Let $\VV$ be a closed, symmetric monoidal category.
The category $\VV_A$ of right $A$-modules is a $\VV$-category.
\end{proposition}

\begin{proof}
For an $A$-module $(M, \rho)$ and an object $K$ in $\VV$, the objects in $\VV$ underlying the tensor product $K\tensor (M,\rho)$ and the cotensor product $(M, \rho)^K$ are  the tensor and cotensor products of the underlying objects in $\VV$.  The right $A$-action on $K\tensor M$ is given by
$$1\tensor \rho\colon K\tensor M\tensor A \rightarrow K\tensor M,$$
while the right action
$$M^K \tensor A \rightarrow M^K$$ 
is adjoint to the composite 
$$K\tensor M^K\tensor A\xrightarrow {\mathrm{ev}\otimes 1} M\tensor A \xrightarrow\rho M.$$
 The forgetful functor $\UU \colon \VV_A\rightarrow \VV$ therefore preserves the tensor and cotensor structures. 
 
Given two $A$-modules  $(M, \rho_{M})$ and $(N, \rho_{N})$, the enrichment $\Map_A(M,N)$ is defined in terms of an equalizer diagram,
$$
\equalizer{\Map_A(M,N)}{\Map(M,N)}{\Map(M\tensor A,N)}{}{},
$$
where the top map is induced by $\rho_{M}\colon M\tensor A\rightarrow M$ and the bottom map is the composite 
$$\Map(M,N)\xrightarrow{-\otimes A} \Map(M\tensor A,N\tensor A)\xrightarrow{(\rho_{N})_{*}} \Map(M\tensor A,N).$$
It is an easy exercise, which we leave it to the reader, to check that these structures are compatible.
\end{proof}

The $\VV$-structure described above interacts well with model category structure, when the monoidal and model category structures are appropriately compatible, e.g., if $\VV$ is a monoidal model category \cite[Definition 3.1]{schwede-shipley}.
Recall that if $\mathscr M$ and $\mathscr N$ are model categories, the model category structure on $\mathscr N$ is \emph{right-induced} by an adjunction
$$\adjunction{\mathscr M}{\mathscr N} LR$$
if the right adjoint $R$ preserves and reflects both weak equivalences and fibrations. 

\begin{theorem}\cite{schwede-shipley} \label{thm:module model}
Let $\VV$ be a symmetric monoidal model category. If $\VV$ is cofibrantly generated and satisfies the monoid axiom, and every object of $\VV$ is small relative to the whole category, then the category $\VV_A$ of right $A$-modules admits a model structure that is right induced from the adjunction
$$\bigadjunction{\VV}{\VV_A}{-\tensor A}{\UU}.$$
\end{theorem}

\begin{remark} The forgetful functor $\UU\colon \VV_A\rightarrow \VV$ is a tensor functor, so it follows from Proposition \ref{prop:v-induced} that, when it exists, the right-induced model structure on $\VV_A$ is $\VV$-structured.
\end{remark}

\begin{remark} It is, of course, also true that the category ${}_{A}\VV$ of left $A$-modules admits a right-induced model category structure under the hypotheses of the theorem above, because ${}_A\VV$ is isomorphic to the category $\VV_{A^{\mathrm{op}}}$ of right modules over the opposite algebra $A^{\mathrm{op}}$.
\end{remark}

\begin{hypothesis}\label{hypo:module} Henceforth, we assume always that $\VV$ is a model category equipped with a symmetric monoidal structure such that for all algebras $A$, the categories  $\VV_{A}$ and ${}_{A}\VV$ of right  and left $A$-modules admit a model category structures such that the forgetful functor in the adjunctions
$$\bigadjunction{\VV}{\VV_A}{-\tensor A}{\UU}$$
and
$$\bigadjunction{\VV}{{}_{A}\VV}{A\tensor -}{\UU}$$
create the weak equivalences in $\VV_{A}$ and ${}_{A}\VV$.
\end{hypothesis}


We recall the definition of tensor products over $A$, and note that they are defined as long as $\VV$ admits reflexive coequalizers.

\begin{definition} Given right and left $A$-modules $M_A$ and ${}_AN$, with structure maps $\rho\colon M\tensor A\rightarrow M$ and $\lambda\colon A\tensor N\rightarrow N$, their \emph{tensor product over $A$} is the object $M\tensor_A N$ in $\VV$ defined by the following coequalizer diagram:
$$
\coequalizer{M\tensor A\tensor N}{M\tensor N}{M\tensor_A N}{\rho\tensor 1}{1\tensor \lambda}.
$$
\end{definition}

The special classes of modules defined below, which are characterized in terms of tensoring over $A$, play an important role in this article.

\begin{definition}\label{defn:special-modules} Let $\VV$ be a symmetric monoidal model category satisfying Hypothesis \ref{hypo:module}.
A left $A$-module $M$ is called
\begin{itemize}
\item \emph{homotopy flat} if $-\tensor_A M\colon\VV_{A} \to \VV$ preserves weak equivalences;
\item \emph{strongly homotopy flat} if it is homotopy flat and for every finite category $\mathsf J$ and every functor $\Phi\colon \mathsf J \to \VV_{A}$, the natural map
$$(\lim_{\mathsf J} \Phi)\otimes_{A} M \to  \lim_{\mathsf J} (\Phi\otimes_{A} M)$$
is a weak equivalence in $\VV$;
\item \emph{homotopy faithful} if $-\tensor_A M\colon\VV_{A} \to \VV$ reflects weak equivalences;
\item \emph{homotopy faithfully flat} if it is both homotopy faithful and strongly homotopy flat;
\item \emph{homotopy projective} if $\Map_A(M,-)\colon {}_A\VV \to \VV$ preserves weak equivalences;
\item\emph{homotopy cofaithful} if $\Map_A(M,-)\colon {}_A\VV \to \VV$ reflects weak equivalences.
\end{itemize}
\end{definition}

There is, of course, an analogous definition for right $A$-modules.

\begin{remark}\label{rmk:retract-faithful} Since a retract of a weak equivalence is a weak equivalence, any $A$-module of which a retract is homotopy faithful (respectively, homotopy cofaithful) is itself homotopy faithful (respectively, homotopy cofaithful).  In particular, if $A$ is a retract of an $A$-module $M$, then $M$ is  homotopy faithful and homotopy cofaithful. For particular choices of base category $\VV$, more can be said about when a module is homotopy faithful or cofaithful; see Remark \ref{rem:classical} for the case of abelian groups, and Proposition \ref{prop:ch-special-modules} for the case of chain complexes.
\end{remark}

\subsection{Bimodules and Quillen adjunctions}
In this section we characterize completely adjunctions between enriched module categories and provide criteria under which these adjunctions are Quillen pairs.

Let $(\VV,\tensor,\kk)$ be a symmetric monoidal model category satisfying Hypothesis \ref{hypo:module}. Given algebras $A$ and $B$ in $\VV$ and a bimodule ${}_A X_B$, there is a $\VV$-adjunction
$$\bigadjunction{\VV_A}{\VV_B}{-\tensor_A X}{\Map_B(X,-)},\quad -\tensor_A X \dashv \Map_B(X,-),$$
where $\VV_{A}$ and $\VV_{B}$ are endowed with the $\VV$-structures of Proposition \ref{prop:Amod-Vstruct}. Let us say that a $\VV$-adjunction,
\begin{equation} \label{eq:v-adjunction}
\adjunction{\VV_A}{\VV_B}{F}{G},\quad F\dashv G,
\end{equation}
is \emph{governed} by a bimodule ${}_A X_B$ if the $\VV$-functors $F$ and $-\tensor_A X$ are isomorphic.

{The following is an enriched version of the classical Eilenberg-Watts theorem \cite{eilenberg}, \cite{watts}.}

\begin{proposition} \label{prop:classification} Let $\VV$ be a symmetric monoidal category that admits all reflexive coequalizers. Every $\VV$-adjunction between $\VV_A$ and $\VV_B$ is governed by an $A$-$B$-bimodule $X$.
\end{proposition}

\begin{proof}
Given a $\VV$-adjunction as in \eqref{eq:v-adjunction}, let $X=F(A)$. A priori, $X$ is only a right $B$-module, but since $F$ is a tensor functor we can endow $X$ with a left $A$-action $\lambda_{}\colon A\tensor X \to X$, equal to the composite
$$
\xymatrix{A\tensor X = A\tensor F(A) \ar[r]^-{\alpha_{A,A}^{-1}} & F(A\tensor A) \ar[r]^-{F(\mu)} & F(A) = X,}
$$
where $\alpha_{K,M}: F(K\otimes M) \xrightarrow \cong K\otimes F(M)$ is the natural isomorphism of Proposition \ref{prop:V-adjunction}, and $\mu:A\tensor A \to A$ is the multiplication map. 
We leave it to reader to check that $X$ is indeed an $A$-$B$-bimodule when endowed with this left $A$-action. 

For any right $A$-module $M$ the canonical isomorphism $M\tensor_A A \cong M$ may be expressed as a coequalizer diagram in $\VV_A$:
$$\xymatrix{M\tensor A\tensor A \ar@<0.5ex>[r]^-{\rho\tensor 1} \ar@<-0.5ex>[r]_-{1\tensor \mu} & M\tensor A \ar[r]^-{\rho_M} & M.}$$
Being a left adjoint, the functor $F$ takes this to a coequalizer diagram in $\VV_B$, which is the top row in the commuting diagram below.
$$
\xymatrix{
F(M\tensor A\tensor A) \ar[d]^-\cong_{\alpha_{M\tensor A, A}} \ar@<0.5ex>[rr]^-{F(\rho\tensor 1)} \ar@<-0.5ex>[rr]_-{F(1\tensor \mu)}&& F(M\tensor A) \ar[d]^-\cong_{\alpha_{M, A}} \ar[r]^-{F(\rho)} & F(M) \ar[d]^-{\therefore \cong} \\
M\tensor A\tensor X \ar@<0.5ex>[rr]^-{\rho\tensor 1}  \ar@<-0.5ex>[rr]_-{1\tensor \lambda} && M\tensor X \ar[r] & M\tensor_A X.
}$$
The two left-hand squares commute because $\alpha$ is a natural transformation and, in the case of the square involving $F(1\otimes \mu_{A})$ and $1\otimes \lambda _{X}$, because $$\alpha_{M\otimes A, A}=(1\otimes \alpha_{A,A})\alpha_{M, A\otimes A}.$$ 
The left and middle vertical maps are isomorphisms because $F$ is a tensor functor. The bottom row is the coequalizer that defines the tensor product $M\tensor_A X$. The desired natural isomorphism 
$F(M) \cong M\tensor_A X$ follows from the universal property of coequalizers.
\end{proof}

When Hypothesis \ref{hypo:module} holds, it is natural to ask when the $\VV$-adjunction between $\VV_A$ and $\VV_B$ governed by a bimodule ${}_AX_B$ is a Quillen pair.  Here is an example of such a situation.

\begin{proposition} \label{prop:x}  Let $\VV$ be a symmetric monoidal model category satisfying Hypothesis \ref{hypo:module}.
Let $A$ and $B$ be algebras in $\VV$ such that fibrations in $\VV_{A}$ and $\VV_{B}$  are created in  $\VV$, and let ${}_AX_B$ be a bimodule. If  $X$ is cofibrant as a right $B$-module, then the adjunction governed by $X$,
\begin{equation} \label{eq:quillen adjunction}
\bigadjunction{\VV_A}{\VV_B}{-\tensor_A X}{\Map_B(X,-)},
\end{equation}
is a Quillen adjunction. The converse holds if the unit $\kk$ is cofibrant in $\VV$.
\end{proposition}

Combining Propositions \ref{prop:classification} and \ref{prop:x}, we obtain the following characterization of enriched adjunctions between module categories that are Quillen pairs.

\begin{corollary}\label{cor:FG-QA} Let $\VV$ be a symmetric monoidal model category satisfying Hypothesis \ref{hypo:module} and such that the unit $\kk$ is cofibrant in $\VV$.
Let $A$ and $B$ be algebras in $\VV$ such that fibrations in $\VV_{A}$ and $\VV_{B}$  are created in  $\VV$.  A $\VV$-adjunction
$$\adjunction{\VV_{A}}{\VV_{B}}FG$$
is a Quillen adjunction if and only if $F(A)$ is cofibrant in $\VV_{B}$.
\end{corollary}

\begin{proof}[Proof of Proposition \ref{prop:x}]
We have a commutative diagram of adjunctions
$$\xymatrix{
\VV_A \ar@<0.5ex>[rr]^-{-\tensor_A X} \ar@<0.5ex>[dd]^-\UU && \VV_B \ar@<0.5ex>[ll] \ar@<-0.5ex>[ddll] \\ \\
\VV \ar@<0.5ex>[uu]^-{-\tensor A} \ar@<-0.5ex>[uurr]_-{-\tensor X}
}$$
Since both the fibrations and the weak equivalences in the module categories are created in  $\VV$, the horizontal adjunction is a Quillen adjunction if and only if the diagonal one is. This in turn happens if and only if for every (trivial) cofibration $i\colon K\rightarrow L$ in $\VV$, the induced map $i\tensor 1\colon K\tensor X\rightarrow L\tensor X$ is a (trivial) cofibration in $\VV_B$. If $X$ is $B$-cofibrant, this condition is satisfied since $\VV_B$ satisfies Axiom \ref{axiom:sm7}. Conversely, if the unit $\kk$ in $\VV$ is cofibrant, then $X\cong \kk\tensor X$ must be cofibrant as a right $B$-module.
\end{proof}

\begin{example} \label{example:morphism}
A morphism of algebras $\vp\colon A \to B$ in $\VV$ induces a natural $A$-$B$-bimodule structure on $B$.  Since $\Map_{B}(B,-)=\vp^{*}\colon \VV_{B}\to \VV_{A}$, the restriction-of-scalars functor, adjunction (\ref{eq:quillen adjunction}) for $X=B$ is exactly the extension/restriction-of-scalars adjunction
\begin{equation}\label{eq:ext-res-adj}
\bigadjunction{\VV_{A}}{\VV_{B}}{-\otimes_{A}B}{\vp^{*}}.
\end{equation}
The right adjoint $\varphi^*$ preserves and reflects weak equivalences by Hypothesis \ref{hypo:module}, so $\varphi^*$ is a right Quillen functor if it also preserves fibrations, which is a somewhat weaker condition than requiring that fibrations of modules are created in $\VV$.  In particular, if fibrations in $\VV_{A}$ and $\VV_{B}$ are created in $\VV$, then $\vp^{*}$ is a right Quillen functor.
\end{example}

\begin{example} \label{ex:morphism2}
The restriction-of-scalars functor $\varphi^*$ also admits a right adjoint,
$$
\bigadjunction{\VV_{B}}{\VV_{A}}{\vp^{*}}{\Map_A(B,-)}.
$$
This adjunction is governed by the bimodule ${}_B B_A$, where $A$ acts on the right via $\varphi$. In particular, if fibrations in $\VV_{A}$ and $\VV_{B}$  are created in  $\VV$, and $B$ is cofibrant as a right $A$-module, then $\varphi^*$ is a left Quillen functor by Proposition \ref{prop:x}. As a special case, note that the category $\VV$ may be identified with the category $\VV_\kk$ of right $\kk$-modules. The unit map $\eta_A\colon \kk \rightarrow A$ governs the forgetful functor $\eta_A^* = \UU\colon \VV_A\rightarrow \VV$. In particular, if $A$ is cofibrant as an object of $\VV$, then all cofibrant $A$-modules are also cofibrant as objects of $\VV$.
\end{example}

\subsection{Dualizable bimodules and Quillen equivalences}\label{sec:dual-bimod}
We now address the question of when the Quillen adjunction governed by a bimodule is a Quillen equivalence.

We begin by analyzing when the restriction-of-scalars adjunction associated to a morphism of algebras induces a Quillen equivalence. To this end, we introduce the concept of a \emph{pure weak equivalence}.

\begin{definition} \label{def:pure we}
A morphism of left $A$-modules $f\colon N\to N'$ is a \emph{pure weak equivalence} if the induced map $1\tensor f\colon M\tensor_A N\rightarrow M\tensor_A N'$ is a weak equivalence for all cofibrant right $A$-modules $M$.
\end{definition}

Under reasonable conditions, all weak equivalences are pure.

\begin{definition}\label{defn:CHF}
We say that $\VV$ satisfies the \emph{CHF hypothesis} if for every algebra $A$ in $\VV$, every cofibrant right $A$-module is homotopy flat (cf. Definition \ref{defn:special-modules}).
\end{definition}

As pointed out in \cite[\S 4]{schwede-shipley}, the CHF hypothesis holds in many monoidal model categories of interest, such as the categories of simplicial sets, symmetric spectra, (bounded or unbounded) chain complexes over a commutative ring, and $S$-modules.

\begin{proposition} \label{prop:pure we alt}
Let $\VV$ be a symmetric monoidal model category satisfying Hypothesis \ref{hypo:module}. If $\VV$ satisfies the CHF hypothesis, then the notions of pure weak equivalence and weak equivalence coincide.
\end{proposition}

\begin{proof}
If all cofibrant right $A$-modules are homotopy flat, then clearly every weak equivalence is pure. Conversely, let $f\colon N\rightarrow N'$ be a pure weak equivalence. We need to show that $f$ is a weak equivalence. We may without loss of generality assume that $N$ and $N'$ are cofibrant. Indeed, by standard model category theory, we can find cofibrant resolutions $q_N\colon QN\to N$ and $q_{N'}\colon QN' \to N'$ and a lift $Qf\colon QN\to QN'$ making the diagram
$$
\xymatrix{QN \ar[r]^-{Qf}_{} \ar[d]^-{q_N}_{\sim} & QN' \ar[d]^-{q_{N'}}_{\sim} \\ N\ar[r]^-f & N'}
$$
commute. Clearly, $f$ is a weak equivalence if and only if $Qf$ is. By tensoring the diagram above from the left with cofibrant (hence homotopy flat) right $A$-modules, one sees that $Qf$ is a pure weak equivalence.

Assume now that $f\colon N\to N'$ is a pure weak equivalence between cofibrant, and hence homotopy flat, $A$-modules. If $q_A\colon QA\to A$ is a cofibrant resolution of $A$ as a right $A$-module, then the commutative diagram
$$
\xymatrix{QA\tensor_A N \ar[r]^-{1\tensor f} \ar[d]^-{q_A\tensor 1} & QA\tensor_A N' \ar[d]^-{q_A\tensor 1} \\ N \ar[r]^-f & N'}
$$
shows that $f$ is a weak equivalence. Indeed, the top horizontal map is a weak equivalence because $f$ is a pure weak equivalence, and the vertical maps are weak equivalences because $N$ and $N'$ are homotopy flat.
\end{proof}

The following result is a slight strengthening of \cite[Theorem 4.3]{schwede-shipley}, based on Example \ref{example:morphism}.

\begin{proposition} \label{prop:resext} Let $\VV$ be a symmetric monoidal model category satisfying Hypothesis \ref{hypo:module}, and let $\varphi\colon A\to B$ be a morphism of algebras in $\VV$ such that $\varphi^{*}: \VV_{B}\to \VV_{A}$ preserves fibrations.  The restriction/\-extension-of-scalars adjunction,
$$\bigadjunction{\VV_A}{\VV_B}{-\tensor_A B}{\varphi^*},$$
is a Quillen equivalence if and only if $\varphi\colon A\rightarrow B$ is a pure weak equivalence of right $A$-modules.
\end{proposition}

\begin{remark}\label{rmk:htpypure}
If all cofibrant modules are homotopy flat, then pure weak equivalences are the same as weak equivalences by Proposition \ref{prop:pure we alt}, so the ``if'' direction of the above proposition recovers \cite[Theorem 4.3]{schwede-shipley}.
\end{remark}

\begin{proof}[Proof of Proposition \ref{prop:resext}]
As explained in Example \ref{example:morphism}, the restriction/\-extension-of-scalars adjunction is a Quillen adjunction because $\vp^{*}$ preserves fibrations. Since weak equivalences are created in the underlying category $\VV$ (by Hypothesis \ref{hypo:module}),  the restriction-of-scalars functor $\varphi^*$ also preserves and reflects all weak equivalences. Therefore, the adjunction is a Quillen equivalence if and only if the unit
$$\eta_{M}:M\to \varphi^{*} (M\otimes_{A}B)$$
is a weak equivalence for all cofibrant right $A$-modules $M$ \cite[Corollary 1.3.16]{hovey-book}. To conclude, note that the morphism in $\VV$ underlying $\eta_{M}$ may be identified with $1\otimes_{A} \vp: M\otimes _{A}A \to M\otimes_{A}B$.   
\end{proof}

We now turn to the question of when the Quillen adjunction governed by a bimodule ${}_AX_B$ induces a Quillen equivalence between $\VV_A$ and $\VV_B$.

\begin{definition}
A bimodule ${}_AX_B$ is called \emph{right dualizable} if there exists a bimodule ${}_BY_A$ together with morphisms
$$u\colon A\to X\tensor_B Y,\quad e\colon Y\tensor_A X\to B,$$
in ${}_A\VV_A$ and ${}_B\VV_B$, respectively, such that the composites
$$X\xrightarrow{u\tensor 1} X\tensor_B Y \tensor_A X  \xrightarrow{1\tensor e} X,\quad Y\xrightarrow{1\tensor u} Y\tensor_A X\tensor_B Y \xrightarrow{e\tensor 1} Y,$$
are the identity maps on $X$ and $Y$, respectively.
\end{definition}

For a right $B$-module $N$, let
$$\ell_N\colon N\tensor_B \Map_B(X,B) \to \Map_B(X,N)$$
be the map of right $A$-modules that is right adjoint to the map
$$N\tensor_B \Map_B(X,B)\tensor_A X \xrightarrow{1\tensor ev} N\tensor_B B \cong N,$$
induced by the evaluation map $ev\colon \Map_B(X,B)\tensor_A X\to B$.
A bimodule ${}_AX_B$ is right dualizable if and only if the map
$$\ell_N\colon N\tensor_B \Map_B(X,B) \to \Map_B(X,N)$$
is an isomorphism for all right $B$-modules $N$ (see Lemma \ref{lem:char-dual}). In particular, this implies that whether or not ${}_AX_B$ is right dualizable depends only on the right $B$-module structure, not on the left $A$-module structure. Moreover, if ${}_AX_B$ is right dualizable, then every right dual ${}_BY_A$ is isomorphic to $\Map_B(X,B)$ as a $B$-$A$-bimodule.

\begin{example}
It is easy to prove that if $\VV$ is the category of abelian groups, then a bimodule ${}_AX_B$ is right dualizable if and only if it is finitely generated and projective as a right $B$-module.
\end{example}

To formulate the homotopical version of the Morita theorem, we will need a homotopical version of dualizability.

\begin{definition}
Let ${}_AX_B$ be a bimodule that is fibrant and cofibrant as a right $B$-module. We call ${}_AX_B$ \emph{homotopy right dualizable} if the natural map
$$\ell_N\colon N\tensor_B \Map_B(X,B) \to \Map_B(X,N)$$
is a weak equivalence for all fibrant and cofibrant right $B$-modules $N$.
\end{definition}

\begin{remark}
If $\VV = \Ch_R$ is the category of (unbounded) chain complexes of modules over a commutative ring $R$, endowed with the projective model structure \cite[Theorem 2.3.11]{hovey-book}, then ${}_AX_B$ is homotopy right dualizable if and only if $X$ is \emph{compact} as an object of the derived category $\mathcal{D}(B)$ (see, e.g., \cite[Theorem A.1]{neeman}). Recall that compact objects in the derived category of a ring correspond to perfect complexes, i.e., bounded complexes of finitely generated projective modules.
\end{remark}


\begin{proposition} \label{prop:derived tensor}
Let $\VV$ satisfy Hypothesis \ref{hypo:module}. If  $\VV$ also satisfies the CHF hypothesis, then for every bimodule ${}_AX_B$, the functor $-\tensor_A X\colon \VV_A\to \VV_B$ preserves weak equivalences between homotopy flat right $A$-modules. In particular, it preserves weak equivalences between cofibrant right $A$-modules.
\end{proposition}

\begin{proof}
Let $f\colon N\rightarrow N'$ be a weak equivalence between homotopy flat right $A$-modules. If $q\colon QX\to X$ is a cofibrant resolution of $X$ as a left $A$-module, then in the diagram
$$
\xymatrix{N\tensor_A QX \ar[r]^-{f\tensor 1} \ar[d]_-{1\tensor q} & N'\tensor_A QX \ar[d]^-{1\tensor q} \\ N\tensor_A X \ar[r]^-{f\tensor 1} & N'\tensor_A X,}
$$
the vertical maps are weak equivalences as $N$ and $N'$ are homotopy flat. The top horizontal map is a weak equivalence because $QX$ is cofibrant, whence homotopy flat. It follows that the bottom horizontal map is a weak equivalence.
\end{proof}

We are now prepared to formulate and prove our homotopical analogue of the classical Morita theorem.

\begin{theorem}[Homotopical Morita theorem] \label{thm:morita}
Let $\VV$ satisfy Hypothesis \ref{hypo:module} and the CHF hypothesis. Let $A$ and $B$ algebras in $\VV$, and let ${}_AX_B$ be a bimodule such that the adjunction 
\begin{equation} \label{eq:quiad2'}
\bigadjunction{\VV_A}{\VV_B}{-\tensor_A X}{\Map_B(X,-)}
\end{equation}
governed by $X$ is a Quillen adjunction. 

If $B$ is fibrant in $\VV_{B}$, and  $X$ is fibrant and cofibrant as a right $B$-module, then the Quillen adjunction (\ref{eq:quiad2'}) 
is a Quillen equivalence if and only if
\begin{enumerate}
\item\label{item:unit} the map $\eta_A\colon A\rightarrow \Map_B(X,X)$ is a weak equivalence; 
\item\label{item:faithful} the bimodule $X$ is homotopy cofaithful as a right $B$-module, i.e., the functor $\Map_B(X,-)$ reflects weak equivalences between fibrant objects; and 
\item\label{item:dualizable} the bimodule $X$ is homotopy right dualizable, i.e., the canonical map
\begin{equation*}
\ell_N\colon N \tensor_B\Map_B(X,B) \to \Map_B(X,N) 
\end{equation*}
is a weak equivalence for all fibrant and cofibrant right $B$-modules $N$.
\end{enumerate}
\end{theorem}

\begin{remark}
The fibrancy hypotheses on $B$ and $X$ are not essential and indeed often trivially fulfilled. They may be removed at the expense of taking derived mapping spaces in the hypotheses.
\end{remark}

\begin{remark}
The theorem above recovers the classical Morita theorem. If $A$ and $B$ are ordinary associative unital rings, then a bimodule ${}_AX_B$ induces an equivalence between $\VV_A$ and $\VV_B$ if and only if $A\cong \Hom_B(X,X)$ and the right $B$-module $X$ is a finitely generated projective generator. A right $B$-module $X$ is finitely generated and projective if and only if it is right dualizable, and it is a generator if and only if the functor $\Hom_B(X,-)$ reflects isomorphisms.  

As a consequence of Proposition \ref{prop:classification}, Theorem \ref{thm:morita} also implies the following broad generalization of \cite[Theorem 1.5]{dugger-shipley}.
\end{remark}

\begin{corollary} Let $\VV$ satisfy Hypothesis \ref{hypo:module} and the CHF hypothesis, and let 
\begin{equation}\label{eqn:FG-adj}
\adjunction{\VV_{A}}{\VV_{B}}FG
\end{equation}
be a $\VV$-adjunction and Quillen adjunction, where $A$ and $B$ are algebras in $\VV$.

If $B$ is fibrant in $\VV_{B}$, and $F(A)$ is fibrant and cofibrant in $\VV_{B}$, then  (\ref{eqn:FG-adj}) is a Quillen equivalence if and only if
\begin{enumerate}
\item the map $\eta_A\colon A\rightarrow \Map_B\big( F(A), F(A)\big)$ is a weak equivalence; 
\item  the functor $\Map_B\big(F(A),-\big)$ reflects weak equivalences between fibrant objects; and 
\item the bimodule $F(A)$ is homotopy right dualizable, i.e., the canonical map
\begin{equation*}
\ell_N\colon N \tensor_B\Map_B\big(F(A),B\big) \to \Map_B\big(F(A),N\big) 
\end{equation*}
is a weak equivalence for all fibrant and cofibrant right $B$-modules $N$.
\end{enumerate}
\end{corollary}

\begin{proof}[Proof of Theorem \ref{thm:morita}]
Suppose first that conditions \eqref{item:unit}, \eqref{item:faithful} and \eqref{item:dualizable} are fulfilled. Since the right adjoint in \eqref{eq:quiad2'} reflects weak equivalences between fibrant objects, we need only to show that the homotopy unit $\widetilde{\eta}_M$ is a weak equivalence for all cofibrant objects $M$ in $\VV_A$ \cite[Corollary 1.3.16]{hovey-book}. Let $r\colon M\tensor_A X\to (M\tensor_A X)^f$ be a fibrant replacement in $\VV_B$, with $r$ a trivial cofibration, and consider the following commutative diagram in $\VV_B$.
$$
\xymatrix{
M\tensor_A A \ar[d]_-\cong \ar[r]_-{(a)}^-{M\tensor_A \eta_A} & M\tensor_A \Map_B(X,X) \ar[d]^-{\ell_M} \\
M \ar[r]^-{\eta_M} \ar[dr]_-{\widetilde{\eta}_M} & \Map_B(X,M\tensor_A X) \ar[d]^-{r_*} && M\tensor_A X\tensor_B \Map_B(X,B) \ar[ull]_-{M\tensor_A \ell_X}^-{(b)} \ar[ll]^-{\ell_{M\tensor_A X}} \ar[d]_-{(c)}^-{r\tensor 1} \\
& \Map_B(X,(M\tensor_A X)^f) && (M\tensor_A X)^f\tensor_B \Map_B(X,B) \ar[ll]_-{\ell_{(M\tensor_A X)^f}}^-{(d)}}
$$
The maps labeled $(a),(b),(c),(d)$ are weak equivalences, for the following reasons.
\begin{itemize}
\item[$(a)$] By our hypothesis \eqref{item:unit}, the map $\eta_A$ is a weak equivalence. Since the right $A$-module $M$ is assumed to be cofibrant, it is also homotopy flat.
\item[$(b)$] Since $X$ is fibrant and cofibrant in $\VV_B$, the map $\ell_X$ is a weak equivalence by \eqref{item:dualizable}. As we pointed out above, $M_A$ is homotopy flat.
\item[$(c)$] Since $M$ is cofibrant, and $-\tensor_A X$ is a left Quillen functor by hypothesis, $M\tensor_A X$ is cofibrant in $\VV_B$. Since $r$ is a weak equivalence and a cofibration with cofibrant source, it is in particular a weak equivalence between two cofibrant objects. Since we assume that all cofibrant modules are homotopy flat, it follows from Proposition \ref{prop:derived tensor} that $r\tensor 1$ is a weak equivalence.
\item[$(d)$] The right $B$-module $(M\tensor_A X)^f$ is fibrant and cofibrant, so the map $\ell_{(M\tensor_A X)^f}$ is a weak equivalence by \eqref{item:dualizable}.
\end{itemize}
It follows that $\widetilde{\eta}_M$ is a weak equivalence.

Conversely, suppose that \eqref{eq:quiad2'} is a Quillen equivalence. Then clearly the right adjoint $\Map_B(X,-)$ reflects weak equivalences between fibrant objects, i.e., \eqref{item:faithful} holds.

Moreover, even though $A$ is not necessarily cofibrant as a right $A$-module, the map $\eta_A$ represents the homotopy unit for $A\in\Ho \VV_A$, because $A$ is homotopy flat, and $X$ is fibrant. Indeed, if $q\colon QA\to A$ is a cofibrant replacement of $A$ in $\VV_A$, then $q$ is a weak equivalence between homotopy flat objects, so $q\tensor_A X\colon QA\tensor_A X\to X$ is a weak equivalence by Proposition \ref{prop:derived tensor}. It follows that we may take $X$ as a fibrant replacement of $QA\tensor_A X$, whence the diagonal map $\widetilde{\eta}_{QA}$ in the commutative diagram
$$
\xymatrix{
QA\ar[d]_-\sim^-q \ar[r]^-{\eta_{QA}} \ar[dr]^-{\widetilde{\eta}_{QA}} & \Map_B(X,QA\tensor_A X) \ar[d]^-{(q\tensor_A X)_*} \\
A\ar[r]^-{\eta_A} & \Map_B(X,X).
}
$$
is a weak equivalence because it represents the homotopy unit for $QA$. Condition \eqref{item:unit} follows immediately.

Finally, we check condition \eqref{item:dualizable}. Let $N$ be a fibrant right $B$-module, 
$$p_N\colon \Map_B(X,N)^c\rightarrow \Map_B(X,N)$$ a cofibrant replacement in $\VV_A$, and $$q\colon {}^cX\rightarrow X$$ a cofibrant replacement in ${}_A\VV$. Consider the following commutative diagram in $\VV$, where the maps labeled by $\sim$ are weak equivalences because cofibrant left or right $A$-modules are homotopy flat.
$$
\xymatrix{
& \Map_B(X,N)^c \tensor_A X \ar[d]_-{p_N\tensor 1} \ar[dr]^-{\widetilde{\epsilon}_N} \\
\Map_B(X,N)^c \tensor_A {}^cX \ar[ur]_-\sim^-{1\tensor q} \ar[dr]^-\sim_-{p_N\tensor 1}
& \Map_B(X,N) \tensor_A X \ar[r]^-{\epsilon_N}
& N \\
& \Map_B(X,N) \tensor_A {}^cX \ar[u]^-{1\tensor q} \ar[ur]_-{f_N}}
$$
The map $\widetilde{\epsilon}_N$ is a map of right $B$-modules and represents the homotopy counit for $N$, so it is a weak equivalence because \eqref{eq:quiad2'} is a Quillen equivalence. It follows that the map $f_N$ is a weak equivalence for every fibrant right $B$-module $N$. In particular, since $B$ is fibrant, the map $f_B$ is a weak equivalence. Note moreover that $f_B$ is a map of left $B$-modules.

Next, let $N$ be a right $B$-module that is both fibrant and cofibrant, and consider the following commutative diagram in $\VV$.
$$
\xymatrix{
N\tensor_B \Map_B(X,B) \tensor_A {}^cX \ar[rr]^-{1\tensor 1\tensor q} \ar[d]_-{\ell_N\tensor_A {}^cX}
&& N\tensor_B \Map_B(X,B)\tensor_A X \ar[r]^-{1\tensor \epsilon_B} \ar[d]_-{\ell_N\tensor_A X}
& N\tensor_B B \ar[d]^-\cong
\\ \Map_B(X,N)\tensor_A {}^cX \ar[rr]^-{1\tensor q}
&& \Map_B(X,N)\tensor_A X \ar[r]^-{\epsilon_N}
& N.
}
$$
The composite of the top horizontal maps is equal to $N\tensor_B f_B$. As we just noted, $f_B$ is a weak equivalence. Since $N_B$ is cofibrant, hence homotopy flat, $N\tensor_B f_B$ is a weak equivalence. On the other hand, the composite of the bottom horizontal maps is equal to $f_N$, which is a weak equivalence by the above, since $N$ is fibrant. We deduce that the left vertical map $\ell_N\tensor_A {}^cX$ is a weak equivalence. Since \eqref{eq:quiad2'} is a Quillen equivalence, the left Quillen functor $-\tensor_A X$ reflects weak equivalences between cofibrant right $A$-modules. It follows that $-\tensor_A{}^cX$ reflects weak equivalences between any right $A$-modules. (Indeed, if $g\colon M\to M'$ is a map in $\VV_A$, then one can take a cofibrant replacement $Qg\colon QM\to QM'$ in $\VV_A$ and argue using the commutative diagram
$$
\xymatrix{
M\tensor_A {}^cX \ar[d] & \ar[l]_-\sim QM\tensor_A {}^cX \ar[d] \ar[r]^\sim & QM\tensor_A X \ar[d] \\
M'\tensor_A {}^cX   & \ar[l]_\sim QM'\tensor_A {}^cX \ar[r]^\sim  & QM'\tensor_A X,
}
$$
observing that $Qg$ is a weak equivalence if and only if $g$ is.)
Thus, $\ell_N$ is a weak equivalence.\end{proof}

\section{Corings and braided bimodules} \label{sec:corings}

Our goal in this section is to study and classify adjunctions between categories of comodules over corings in symmetric monoidal categories admitting appropriate limits and colimits.
We begin by recalling the elementary theory of corings and their comodules, then introduce the notion of \emph{braided bimodules} and show that every adjunction between categories of comodules over corings, relative to a fixed adjunction between the underlying module categories, is governed by a braided bimodule.

Throughout this section, $(\VV,\tensor,\kk)$ denotes a symmetric monoidal category that admits all reflexive coequalizers and coreflexive equalizers.

\subsection{Corings and their comodules}
If $A$ is an algebra in $\VV$, then the tensor product $-\tensor_A -$ endows the category of $A$-bimodules ${}_A\VV_A$ with a (not necessarily symmetric) monoidal structure. The unit is $A$, viewed as an $A$-bimodule over itself.
\begin{definition}
An \emph{$A$-coring} is a coalgebra in the monoidal category $({}_A\VV_A,\tensor_A,A)$, i.e., an $A$-bimodule $C$ together with maps of $A$-bimodules $\Delta\colon C \rightarrow C\tensor_A C$ and $\epsilon\colon C\rightarrow A$, such that the diagrams
$$
\xymatrix{C \ar[r]^-{\Delta} \ar[d]_-{\Delta} & C\tensor_A C \ar[d]^-{1\tensor \Delta} \\ C\tensor_A C \ar[r]^-{\Delta\tensor 1} & C\tensor_A C \tensor_A C} \quad \quad \xymatrix{C\ar[r]^-{\Delta} \ar[d]_-{\Delta} \ar@{=}[dr] & C\tensor_A C \ar[d]^-{1\tensor \epsilon} \\ C\tensor_A C\ar[r]_-{\epsilon\tensor 1} & C}
$$
are commutative. Here, we tacitly make the identifications $A\tensor_A C = C = C\tensor_A A$ in the lower right corner. A \emph{morphism of $A$-corings} is a map of $A$-bimodules $f\colon C\rightarrow D$ such that the diagrams
$$\xymatrix{C\ar[r]^-{\Delta_C} \ar[d]_-f & C\tensor_A C\ar[d]^-{f\tensor_A f} \\ D \ar[r]^-{\Delta_D} & D\tensor_A D}\quad \quad \quad  \xymatrix{C\ar[r]^-{\epsilon_C} \ar[d]_-f & A \ar@{=}[d] \\ D \ar[r]^-{\epsilon_D} & A}$$
commute.
\end{definition}

We need to allow morphisms between corings to change the algebra as well. To this end, note that if $\varphi\colon A\rightarrow B$ is a morphism of algebras, then there is an extension/restriction-of-scalars adjunction,
$$\bigadjunction{{}_A\VV_A}{{}_B\VV_B}{\varphi_*}{\varphi^*},\quad \varphi_* \dashv \varphi^*,$$
where $\varphi_*(M) = B\tensor_A M \tensor_A B$. Moreover, $\varphi_{*}$ is an op-monoidal functor, i.e., there is a natural transformation
$$\varphi_*(M\tensor_A N) \rightarrow \varphi_*(M)\tensor_B \varphi_*(N),$$
and a morphism $\varphi_*(A) \to B$, which allow us to endow $\varphi_*(C)$ with the structure of a $B$-coring whenever $C$ is an $A$-coring.

\begin{definition} \label{def:coring}
A \emph{coring} in $\VV$ is a pair $(A,C)$ where $A$ is an algebra in $\VV$ and $C$ is an $A$-coring. A \emph{morphism of corings} $(A,C)\rightarrow (B,D)$ is a pair $(\varphi,f)$ where $\varphi\colon A\rightarrow B$ is a morphism of algebras and $f\colon \vp_*(C)\rightarrow D$ is a morphism of $B$-corings. The category of corings in $\VV$ is denoted $\Coring_{\VV}$.
\end{definition}

\begin{remark}
There is no natural $A$-coring structure on $\varphi^*(D)$ in general, but if we let $f^\sharp\colon C\rightarrow \vp^*(D)$ denote the adjoint of $f\colon \vp_*(C)\to D$, then the condition that $f$ is a morphism of $B$-corings is equivalent to saying that the diagrams of $A$-bimodules
\begin{equation} \label{eq:flat}
\xymatrix{C\ar[r]^-{\Delta_C} \ar[d]_-{f^\sharp} & C\tensor_A C\ar[d]^-{f^\sharp\tensor_\varphi f^\sharp} \\ \varphi^*(D) \ar[r]^-{\Delta_D} & \varphi^*(D\tensor_B D)}\quad \quad \quad  \xymatrix{C\ar[r]^-{\epsilon_C} \ar[d]_-{f^\sharp} & A \ar[d]^-\varphi \\ \varphi^*(D) \ar[r]^-{\epsilon_D} & \varphi^*(B)}
\end{equation}
are commutative.
\end{remark}

\begin{definition} Let $(A,C)$ be a coring in $\VV$. 
A right \emph{$(A,C)$-comodule} is a right $A$-module $M$ together with a morphism of right $A$-modules
$\delta\colon M\rightarrow M\tensor_A C$ such that the diagrams
$$
\xymatrix{M \ar[r]^-{\delta} \ar[d]_-{\delta} & M\tensor_A C \ar[d]^-{1\tensor \Delta} \\ M\tensor_A C \ar[r]^-{\delta\tensor 1} & M\tensor_A C \tensor_A C} \quad \quad \xymatrix{M\ar[r]^-{\delta} \ar@{=}[dr] & M\tensor_A C \ar[d]^-{1\tensor \epsilon} \\ & M}
$$
are commutative. A \emph{morphism of $(A,C)$-comodules} is a morphism $f\colon M\rightarrow N$ of right $A$-modules such that the diagram
$$
\xymatrix{M \ar[r]^-{\delta_M} \ar[d]_-f & M\tensor_A C \ar[d]^-{f\tensor 1} \\ N \ar[r]^-{\delta_N} & N\tensor_A C}
$$
commutes.
\end{definition}

We let $\VV_A^C$ denote the category of right $(A,C)$-comodules. There is an adjunction
$$\bigadjunction{\VV_A^C}{\VV_A}{\UU_A}{-\tensor_A C},\quad \UU_A \dashv -\tensor_A C,$$
where $\UU_A$ is the forgetful functor.  The category ${}_{A}^{C}\VV$ of left $(A,C)$-comodules is defined analogously.

We now give some examples of corings.

\begin{example}[Trivial coring]
For every algebra $A$, we can form the \emph{trivial coring} $(A,A)$. The comultiplication $A\to A\tensor_A A$ is the natural isomorphism, and the counit $A\to A$ is the identity map. The forgetful functor $\UU_A\colon \VV_A^A\to \VV_A$ is an isomorphism of categories.
\end{example}

\begin{example}[Coalgebras]
For every coalgebra $C$ in $\VV$, there is a coring $(\kk,C)$. The category $\VV_\kk^C$ is isomorphic to the category of comodules over $C$.
\end{example}

The examples above show that the study of comodules over corings englobes the study of modules over algebras and comodules over coalgebras.

\begin{example}[Descent coring] \label{ex:desc-coring}
Let $\varphi\colon A\to B$ be a morphism of algebras in $\VV$. The \emph{descent coring} associated to $\vp$ is the coring $(B,B\tensor_A B)$, where the comultiplication is the composite
$$B\tensor_A B \cong B\tensor_A A \tensor_A B \xrightarrow{1\tensor \varphi \tensor 1} B\tensor_A B\tensor_A B \cong \big( B\tensor_A B \big) \tensor_B \big( B\tensor_A B \big),$$
and the counit $B\tensor_A B\to B$ is induced by the multiplication in $B$.

There is a morphism of corings $(\varphi,f)\colon (A,A)\to (B,B\tensor_A B)$, where $f$ is the identity map on $B\tensor_A B$.
\end{example}

Further important examples of corings in arise in the theory of Hopf-Galois extensions \cite[\S34]{brzezinski-wisbauer}. Hopf-Galois extensions for structured ring spectra were introduced by Rognes \cite{rognes}, and his framework was later generalized by Hess \cite{hess:hhg}. Another source of corings is provided by Hopf algebroids (see \cite[\S31.6]{brzezinski-wisbauer} or \cite[Appendix 1]{ravenel}). Every Hopf algebroid has an underlying coring, obtained by forgetting the left and right units and the antipode. Comodules over the Hopf algebroid (see e.g.~\cite[Definition A1.1.2]{ravenel}) are the same thing as comodules over the underlying coring.

\begin{proposition} \label{prop:comodule v-category}
Let $\VV$ be a closed, symmetric monoidal category that admits all reflexive coequalizers and coreflexive equalizers. The category $\VV_A^C$ of $(A,C)$-comodules is a $\VV$-category if it admits all coreflexive equalizers.
\end{proposition}

\begin{remark}\label{rmk:coref-eq}
If $\VV$ is locally presentable, then $\VV_{A}$ is locally presentable, as it is the category of algebras for the monad on $\VV$ with underlying functor $-\otimes A$ \cite [2.78]{adamek-rosicky}. This implies in turn that $\VV_A^C$ is locally presentable, and therefore complete, since it is the category of coalgebras for the comonad on $\VV_{A}$ with underlying functor $-\otimes_{A}C$  \cite [Proposition A.1]{ching-riehl}.  In particular, if $\VV$ is locally presentable, then $\VV_A^C$ admits all coreflexive equalizers.

On the other hand, by the dual of \cite[Corollary 3]{linton}, if $-\otimes_{A}C: \VV_{A}\to \VV_{A}$ preserves coreflexive equalizers, then $\VV_{A}^{C}$ admits all coreflexive equalizers.
\end{remark}

\begin{proof}
For $K\in \VV$ and $M\in \VV_A^C$, the tensor product $K\tensor M\in \VV_A^C$ is defined as the tensor product of the underlying objects in $\VV$, together with the evident right $A$-module and $C$-comodule structures. That $\VV_A^C$ is cotensored over $\VV$ is ensured by the (dual) Adjoint Lifting Theorem \cite[\S4.5]{borceux}, which we can apply to the diagram
\begin{equation*}
\xymatrix{
\VV_A^C \ar@<1ex>[rr]^-{K\tensor -} \ar@<-1ex>[dd]_-{\UU_A} && \VV_A^C \ar@<1ex>@{-->}[ll]^-{(-)^K} \ar@<1ex>[dd]^-{\UU_A} \\ \\
\VV_A \ar@<-1ex>[rr]_{K\tensor -} && \VV_A \ar@<-1ex>[ll]_-{(-)^K}}
\end{equation*}
because $\UU_{A}$ and $\UU_{B}$ are comonadic, and $\VV_A^C$ admits coreflexive equalizers by hypothesis. Explicitly, the cotensor product $M^K$ can be defined as the equalizer of the following diagram in $\VV_A^C$:
$$\Map_\VV(K,M)\tensor_A C \rightrightarrows \Map_\VV(K,M\tensor_A C) \tensor_A C.$$
The top map is induced by $\delta_{M}: M \to M\otimes_{A}C$, and the bottom map is given by
{\small $$\xymatrix{\Map_\VV(K,M)\tensor_A C \ar[r]^-{1\tensor \Delta} & \Map_\VV(K,M)\tensor_A C\tensor_A C \ar[r]^-{\nu\tensor 1} & \Map_\VV(K,M\tensor_A C)\tensor_A C,}$$}
where $\nu\colon \Map_\VV(K,M)\tensor_A C\rightarrow \Map_\VV(K,M\tensor_A C)$ is adjoint to the map
$$\xymatrix{K\tensor \Map_\VV(K,M)\tensor_A C \ar[r]^-{ev\tensor 1} & M\tensor_A C.}$$

Similarly, existence of the $\VV$-enrichment $\Map_A^C(M,N) \in \VV$ is ensured by applying the (dual) Adjoint Lifting Theorem to the following diagram.
\begin{equation*}
\xymatrix{
\VV \ar@<1ex>[rr]^-{-\tensor M} \ar@<-1ex>@{=}[dd] && \VV_A^C \ar@<1ex>@{-->}[ll]^-{\Map_A^C(M,-)} \ar@<1ex>[dd]^-{\UU_A} \\ \\
\VV \ar@<-1ex>[rr]_{-\tensor M} && \VV_A \ar@<-1ex>[ll]_-{\Map_A(M,-)}}
\end{equation*}
Explicitly, $\Map_A^C(M,N)$ is the equalizer of the following diagram in $\VV$:
$$\Map_A(M,N) \rightrightarrows \Map_A(M,N\tensor_A C).$$
The top map is $(\delta_N)_* \colon \Map_A(M,N)\rightarrow \Map_A(M,N\tensor_A C)$, and the bottom map is the composite
$$\xymatrix{\Map_A(M,N) \ar[r]^-{-\tensor_A C} & \Map_A(M\tensor_A C,N\tensor_A C) \ar[r]^-{\delta_M^*} & \Map_A(M,N\tensor_A C).}$$
\end{proof}

\begin{remark} The forgetful functor $\UU_A$ is clearly a tensor functor, so the adjunction
$$\bigadjunction{\VV_A^C}{\VV_A}{\UU_A}{-\tensor_A C},\quad \UU_A \dashv -\tensor_A C,$$
is a $\VV$-adjunction with respect to the structures defined in the proof above, by Proposition \ref{prop:V-adjunction}.
\end{remark}

\subsection{Braided bimodules} \label{subsec:braided bimodules}
By Proposition \ref{prop:classification}, every $\VV$-adjunction $(F,G)$ between $\VV_A$ and $\VV_B$ is governed by a bimodule ${}_AX_B$. Our next goal is to investigate what extra structure on $X$ is needed to lift $(F,G)$ to a $\VV$-adjunction $(\widetilde{F},\widetilde{G})$ between $\VV_A^C$ and $\VV_B^D$, such that the diagram of left adjoints 
\begin{equation} \label{eq:relative lefta adjoints}
\xymatrix{\VV_A^C \ar[r]^-{\widetilde{F}} \ar[d]^-{\UU_A} & \VV_B^D \ar[d]^-{\UU_B} \\ \VV_A \ar[r]^-F & \VV_B}
\end{equation}
commutes up to natural isomorphism. To this end, we introduce the notion of a \emph{braided bimodule}.

\begin{definition} \label{def:braided bimodule}  Let $(A,C)$ and $(B,D)$ be corings in a monoidal category $\VV$.
A \emph{braided $(A,C)$-$(B,D)$-bimodule} is a pair $(X,T)$ where $X$ is an $A$-$B$-bimodule and $T$ is a morphism of $A$-$B$-bimodules
$$T\colon C\tensor_A X\rightarrow X\tensor_B D$$
satisfying the following axioms.

\noindent {\bf (Pentagon axiom)}

The diagram
\begin{equation} \label{eq:pentagon}
\xymatrix{C\tensor_A X \ar[d]_-{\Delta_C\tensor 1} \ar[rr]^-{T} && X\tensor_B D \ar[d]^-{1\tensor \Delta_D} \\
C\tensor_A C\tensor_A X \ar[dr]_-{1\tensor T} && X\tensor_B D \tensor_B D \\
& C\tensor_A X\tensor_B D \ar[ur]_-{T\tensor 1}}
\end{equation}
commutes.

\vskip5pt
\noindent {\bf (Counit axiom)}

The diagram
\begin{equation} \label{eq:counit}
\xymatrix{C\tensor_A X \ar[rr]^-{T} \ar[d]^-{\epsilon_C\tensor 1} && X\tensor_B D \ar[d]^-{1\tensor \epsilon_D} \\
A\tensor_A X \ar[r]^-\cong & X & X\tensor_B B \ar[l]_-\cong}
\end{equation}
commutes.

We write $(X,T): (A,C) \to (B,D)$ to indicate that $(X,T)$ is a braided $(A,C)$-$(B,D)$-bimodule.
\end{definition}

\begin{definition}
A \emph{morphism $(X,T)\rightarrow (X',T')$ of braided bimodules} is a morphism of bimodules $f\colon {}_AX_B\rightarrow {}_AX'_B$ such that the diagram
$$\xymatrix{C\tensor_A X \ar[r]^-{T} \ar[d]_-{1\tensor f} & X\tensor_B D \ar[d]^-{f\tensor 1} \\ C\tensor_A X' \ar[r]^-{T'} & X'\tensor_B D}$$
commutes.
\end{definition}

\begin{remark}
The notion of a braided bimodule does not seem to have appeared in the literature before. The closest we have found is the notion of an \emph{entwining structure} (see \cite[\S32]{brzezinski-wisbauer}). The two notions are related as follows: a triple $(A,C)_\psi$ is an entwining structure if and only if $X= {}_\kk A_\kk$ and $T = \psi\colon C\tensor A\rightarrow A\tensor C$ define a braided bimodule from the coring $(\kk,C)$ to itself such that $\psi$ is a morphism of right $A$-modules and $\psi\circ (1\tensor \eta_A) = \eta_A\tensor 1$.
\end{remark}

Just as one may form the bicategory $\ALG_\VV$ of algebras and bimodules (cf.~\cite[(2.5)]{benabou}), we may define a bicategory $\CORING_\VV$ of corings and braided bimodules.

\begin{definition} \label{def:coring bicat}
The bicategory $\CORING_\VV$ has as objects corings $(A,C)$ in $\VV$. A $1$-morphism from $(A,C)$ to $(B,D)$ is a braided bimodule $(X,T):(A,C) \to (B,D)$, while a $2$-morphism is a morphism of braided bimodules, as in Definition \ref{def:braided bimodule}.

The composition of $1$-morphisms is given by tensoring, i.e., the composite of $(X,T):(A,C) \to (A',C')$ and $(X',T'):(A',C') \to (A'',C'')$ is the braided bimodule 
$$\big(X\otimes _{A'}X', (1\otimes T')(T\otimes1)\big):(A,C)\to (A'', C'').$$
The composition of $2$-morphisms is simply the usual composition of morphisms of $A$-$B$-bimodules.
\end{definition}

It is a straightforward exercise to prove first that $\big(X\otimes _{A'}X', (1\otimes T')(T\otimes1)\big)$ is indeed a braided bimodule and then that $\CORING_{\VV}$ does satisfy the axioms of a bicategory.  Note that forgetting the corings and braidings defines a bifunctor 
\begin{equation}\label{eqn:coring-to-alg}\CORING_{\VV}\to \ALG_{\VV}\end{equation} 
(cf.~Example \ref{ex:bimod}).

Given an algebra $A$, the \emph{trivial coring} is $A$ itself, with structure maps the natural isomorphisms. Every bimodule ${}_AX_B$ may be viewed as a braided bimodule between the trivial corings $(X,T)\colon (A,A)\rightarrow (B,B)$, where the braiding is the natural isomorphism $T\colon A\tensor_A X \cong X\tensor_B B$. This defines a bifunctor $\ALG_\VV\rightarrow \CORING_\VV$, which is a section of the bifunctor (\ref{eqn:coring-to-alg}).

\begin{remark} \label{rem:C=A}
If $C$ is the trivial $A$-coring $C=A$, then a braided bimodule $(X,T)\colon (A,A)\rightarrow (B,D)$ is the same thing as an $A$-$B$-bimodule $X$ together with a compatible right $D$-comodule structure on $X$.
\end{remark}

\begin{proposition} \label{prop:bbadj}
Let $\VV$ be a closed, symmetric monoidal category that admits all reflexive coequalizers and coreflexive equalizers.  If $(A,C)$ is a coring in $\VV$ such that $\VV_A^C$ admits all coreflexive equalizers, then
every braided bimodule $(X,T):(A,C)\to (B,D)$ gives rise to a $\VV$-adjunction
$$\bigadjunction{\VV_A^C}{\VV_B^D}{T_{*}}{T^{*}},\quad T_{*} \dashv T^{*},$$
such that the diagram of left adjoints,
$$\xymatrix{
\VV_A^C \ar[rr]^-{T_{*}} \ar[d]_-{\UU_A} && \VV_B^D \ar[d]^-{\UU_B} \\
\VV_A \ar[rr]^{-\tensor_A X} && \VV_B,}
$$
commutes.
\end{proposition}

\begin{proof}
Given $M\in \VV_A^C$, the braiding $T$ allows us to define a right $D$-comodule structure on the right $B$-module $M\tensor_A X$ to be the composite
\begin{equation} \label{eq:comodule}
\xymatrix{M\tensor_A X \ar[r]^-{\delta_M\tensor 1} & M\tensor_A C\tensor_A X \ar[r]^-{1\tensor T} & M\tensor_A X\tensor_B D.}
\end{equation}
Axioms \eqref{eq:pentagon} and \eqref{eq:counit} ensure that this morphism endows $M\tensor_A X$ with the structure of a $D$-comodule. More precisely, the commutativity of the diagram
{\small$$
\xymatrix@C=.5cm{
M\tensor_A X \ar[d]_-{\delta_M\tensor 1} \ar[rr]^-{\delta_M\tensor 1} &&
M\tensor_A C\tensor_A X \ar[d]^-{1\tensor \Delta_C\tensor 1} \ar[rr]^-{1\tensor T} &&
M\tensor_A X\tensor_B D \ar[dd]^-{1\tensor 1 \tensor \Delta_D} \\
M\tensor_A C \tensor_A X \ar[d]_-{1\tensor T} \ar[rr]^-{\delta_M\tensor 1\tensor 1} &&
M\tensor_A C \tensor_A C \tensor_A X \ar[d]^-{1\tensor T} &
M\tensor_A\left(\ref{eq:pentagon}\right) \\
M\tensor_A X\tensor_B D \ar[rr]_-{\delta_M\tensor 1\tensor 1} &&
M\tensor_A C\tensor_A X\tensor_B D \ar[rr]_{1\tensor T\tensor 1} &&
M\tensor_A X\tensor_B D\tensor_B D}
$$}\\
implies that the $D$-coaction \eqref{eq:comodule} is coassociative if and only if diagram \eqref{eq:pentagon} commutes after applying the functor $M\tensor_A -$ to it. Similarly, the commutativity of the diagram
$$\xymatrix{M\tensor_A X \ar[ddr]_-{\cong} \ar[r]^-{\delta_M\tensor 1} & M\tensor_A C \tensor_A X \ar[dd]_-{1\tensor \epsilon_C\tensor 1} \ar[rr]^-{1\tensor T} && M\tensor_A X\tensor_B D \ar[dd]^-{1\tensor 1 \tensor \epsilon_D} \\
&& M\tensor_A \left( \ref{eq:counit} \right) \\
& M\tensor_A A \tensor_A X \ar[r]^-{\cong} & M\tensor_A X & M\tensor_A X\tensor_B B \ar[l]_-\cong
}$$
implies that the $D$-coaction \eqref{eq:comodule} is counital if and only if the diagram \eqref{eq:counit} commutes after applying $M\tensor_A -$ to it. We can therefore set
$$T_{*}(M, \delta_{M})= \big(M\otimes _{A} X, (1\otimes T)(\delta_{M}\otimes 1) \big).$$

The existence of the right adjoint $T^{*}$ is ensured by the (dual) Adjoint Lifting Theorem, since $\UU_{A}$ and $\UU_{B}$ are comonadic, and $\VV_A^C$ admits all coreflexive equalizers by hypothesis. For $M\in \VV_B$ the value of $T^{*}$ at the cofree $D$-comodule $M\tensor_B D$ is the cofree $C$-comodule:
$$T^{*}(M\tensor_B D) = \Map_B(X,M)\tensor_A C.$$
In particular, $T^{*}(D)= \Map_B(X,B)\tensor_A C.$

In general, the right adjoint $T^{*}(N)$ can be calculated as the equalizer 
\begin{equation} \label{eq:R_X}
\Map_B(X,N)\tensor_A C \rightrightarrows \Map_B(X,N\tensor_B D)\tensor_A C
\end{equation}
where the top map is
$$(\delta_N)_*\tensor 1 \colon \Map_B(X,N)\tensor_A C \rightarrow \Map_B(X,N\tensor_B D)\tensor_A C,$$
and the bottom map is the composite
\begin{align*}
\Map_B(X,N)\tensor_A C & \xrightarrow{(-\tensor_B D)\tensor 1}\Map_B(X\tensor_B D,N\tensor_B D)\tensor_A C \\
& \xrightarrow {T^*\tensor 1}\Map_B(C\tensor_A X,N\tensor_B D)\tensor_A C \\
& \xrightarrow {1\tensor \Delta_C} \Map_B(C\tensor_A X,N\tensor_B D)\tensor_A C\tensor_A C \\
& \xrightarrow{g\tensor 1} \Map_B(X,N\tensor_B D)\tensor_A C,
\end{align*}
where the map $g \colon \Map_B(C\tensor_A X,N\tensor_B D)\tensor_A C \rightarrow \Map_B(X,N\tensor_B D)$ is the adjoint to the evaluation map
$$\Map_B(C\tensor_A X, N\tensor_B D)\tensor_A C\tensor_A X \rightarrow N\tensor_B D.$$
\end{proof}

\begin{remark}
We note here for later reference that, for any right $B$-module $M$, the $M\otimes_{B}D$-component of the counit of the $T_{*}\dashv T^{*}$ adjunction is given by the composite
$$\Map_{B}(X,M) \otimes_{A}C \otimes_{A}X \xrightarrow {1\otimes T} \Map_{B}(X,B) \otimes_{A}X \otimes_{B}D\xrightarrow{\mathrm{ev}\otimes 1} M\otimes _{B}D.$$
\end{remark}

\begin{remark}\label{rmk:interpret}
If $C$ is coaugmented, then $A$ is a $C$-comodule, and by plugging in $M = A$, the argument above shows that axioms \eqref{eq:pentagon} and \eqref{eq:counit} are equivalent to saying that \eqref{eq:comodule} defines a $D$-comodule structure on $M\tensor_A X$ for every $M\in \VV_A^C$. Note also that \eqref{eq:pentagon} may be interpreted as saying that $T$ is a morphism of right $D$-comodules, when $C\tensor_A X$ is given the $D$-comodule structure \eqref{eq:comodule}.
\end{remark}

\begin{remark}
There is a natural bijection between $C$-$D$-braidings $T$ on a bimodule ${}_AX_B$ and $(A,C)$-$(B,D)$-bicomodule structures $\delta$ on $C\tensor_A X$ that extend the given left $(A,C)$-comodule and right $B$-module structures. Indeed, given $T$ we may define $\delta$ as the composite
$$C\tensor_A X \xrightarrow{\Delta_C \tensor 1} C\tensor_A C\tensor_A X \xrightarrow{1\tensor T} C\tensor_A X\tensor_B D .$$
Conversely, given $\delta$ we may define $T$ as the composite
$$C\tensor_A X \xrightarrow{\delta} C\tensor_A X\tensor_B D\xrightarrow{\epsilon_C\tensor 1\tensor 1} X\tensor_B D.$$
\end{remark}

By Proposition \ref{prop:bbadj}, every braided bimodule $(X,T):(A,C)\to (B,D)$ gives rise to a $\VV$-adjunction between $\VV_A^C$ and $\VV_B^D$ relative to the $\VV$-adjunction between $\VV_A$ and $\VV_B$ governed by $X$. In fact, all relative $\VV$-adjunctions arise in this way. Before establishing the general case, we examine some important special cases.

\begin{example}[Forgetful functor] \label{ex:braided bimodules 1}
Recall that for any algebra $A$ in $\VV$, the canonical isomorphism $A\rightarrow A\tensor_A A$ and the identity map on $A$ make $(A,A)$ into a coring, the \emph{trivial coring}. Moreover, the forgetful functor $\UU_A\colon \VV_A^A \rightarrow \VV_A$ is an isomorphism of categories. For any coring $(A,C)$, with counit $\epsilon_C\colon C \to A$, the adjunction
$$\bigadjunction{\VV_A^C}{\VV_A = \VV_A^A}{\UU_A}{-\tensor_A C}$$
is governed by the braided bimodule $(X,T)\colon(A,C)\to (A,A)$, where $X = A$ and $T=\epsilon_C \colon C\rightarrow A$, i.e., $\UU_{A}=(\epsilon_C)_{*}$, and $-\otimes_{A}C= \epsilon_C^{*}$.
\end{example}

\begin{example}[Morphisms of corings] \label{ex:braided bimodules 2}
Just as morphisms of algebras give rise to bimodules (cf.~Example \ref{example:morphism}), morphisms of corings give rise to braided bimodules. The braided bimodule associated to a morphism of corings,
$$(\varphi,f)\colon (A,C)\rightarrow (B,D),$$ 
has underlying $A$-$B$-bimodule ${}_AB_B$, where $A$ acts on $B$ through $\varphi$. The braiding 
$$T_{\varphi, f}\colon C\tensor_A B \rightarrow B\tensor_B D \cong D$$ 
is defined to be the composite
$$C\tensor_A B \cong A\tensor_A C\tensor_A B \xrightarrow{\varphi\tensor 1\tensor 1} B\tensor_A C\tensor_A B \xrightarrow{f} D.$$

We have thus constructed a braided bimodule $(B, T_{\vp, f})\colon (A,C) \to (B,D)$, inducing an adjunction
$$\bigadjunction{\VV_{A}^{C}}{\VV_{B}^{D}}{(T_{\vp,f})_{*}}{(T_{\vp,f})^{*}},$$
as long as $\VV_{A}^{C}$ admits all coreflexive equalizers. Observe that it is \textbf{not} necessary for the monoidal structure on $\VV$ to be closed in order for this adjunction to exist, as we are lifting the adjunction
$$\bigadjunction{\VV_{A}^{}}{\VV_{B}^{}}{-\otimes_{A}B}{\vp^{*}},$$
which exists even if $\VV$ is not closed, unlike when $X\not=B$. Note also that the $D$-component of the counit of the adjunction $(T_{\vp,f})_{*}\dashv (T_{\vp,f})^{*}$ may be identified with the morphism
$$f:B \otimes_{A}C \otimes_{A}B\to D.$$
\end{example}

\begin{example}[Change of corings] \label{ex:braided bimodules 3}
When $A=B$ and  $\varphi=1_A\colon A \to A$ in the example above, the braiding $T_{1_A,f}: C\otimes_{A}A\to A \otimes _{A} D$ is nothing but the morphism of $A$-corings $f:C\to D$, up to isomorphism in the source and target. We denote the induced adjunction
\begin{equation}\label{eq:corestriction}
\bigadjunction{\VV_{A}^{C}}{\VV_{A}^{D}}{f_{*}}{f^{*}}
\end{equation}
and call it the \emph{coextension/corestriction-of-scalars adjunction} or \emph{change-of-corings adjunction} associated to $f$.  Note that the $D$-component of the counit of the $f_{*}\dashv f^{*}$ adjunction  is $f$ itself and that for every $(A,C)$-comodule $(M, \delta)$,
$$f_{*}(M,\delta)= \big(M, (1\otimes f)\delta\big).$$
\end{example}

We will now establish the general case.

\begin{theorem} \label{thm:relative classification}
Let $\VV$ be a closed symmetric monoidal category admitting all reflexive coequalizers and coreflexive equalizers.

If $(A,C)$ is a coaugmented coring such that  $\VV_A^C$ admits all coreflexive equalizers, and $(B,D)$ is any coring, then every $\VV$-adjunction between $\VV_A^C$ and $\VV_B^D$, relative to a $\VV$-adjunction between $\VV_A$ and $\VV_B$, is governed by a braided bimodule.
\end{theorem}

\begin{proof}
Consider a relative adjunction $(\widetilde{F},\widetilde{G})$, as in \eqref{eq:relative lefta adjoints}. By Proposition \ref{prop:classification} the underlying adjunction $F\colon \VV_A \leftrightarrows \VV_B: G$ is governed by a bimodule ${}_AX_B$. We have to construct a braiding $T\colon C\otimes _{A}X\to X\otimes_{B}D$ and show that $(X,T)\colon(A,C)\to (B,D)$ governs the adjunction we started with. 

Since the diagram  \eqref{eq:relative lefta adjoints} of left adjoints commutes, for every $C$-comodule $(M,\delta)$, the $B$-module underlying $\widetilde{F}(M, \delta)$ is isomorphic to $M\tensor_A X$. Let
$$\tilde\delta \colon M\tensor_A X \rightarrow M\tensor_A X \tensor_B D$$
denote the right $D$-comodule structure on $\widetilde{F}(M, \delta)$. Define $T\colon C\tensor_A X\rightarrow X\tensor_B D$ to be the composite
$$C\tensor_A X \xrightarrow{\widetilde \Delta_{}} C\tensor_A X\tensor_B D \xrightarrow{\epsilon_C \tensor 1 \tensor 1} A\tensor_A X\tensor_B D \cong X\tensor_B D,$$
where $\Delta : C\to C\otimes _{A}C$ is the comultiplication on $C$, seen as a right $C$-coaction on $C$, and $\epsilon\colon C\rightarrow A$ is the counit of $C$. We have to verify axioms \eqref{eq:pentagon} and \eqref{eq:counit} and show that the comodule structure $\tilde\delta$ on $M\tensor_A X$ agrees with the one induced from $T$ as in \eqref{eq:comodule}.

To check this last condition, consider the diagram
\begin{equation} \label{eq:M}
\xymatrix{M\tensor_A X \ar[r]^-{\tilde\delta} \ar[d]_-{\delta\tensor 1} & M\tensor_A X \tensor_B D \ar[d]_-{\delta\tensor 1\tensor 1} \ar@{=}[drr] \\
M\tensor_A C\tensor_A X \ar[r]_-{1\tensor \tilde \Delta_{}} & M\tensor_A C\tensor_A X \tensor_B D \ar[rr]_-{1\tensor \epsilon\tensor 1 \tensor 1} && M\tensor_A X\tensor_B D.}
\end{equation}
Commutativity of the left square is equivalent to the fact that $\widetilde{F}(\delta)$ is a morphism of $D$-comodules, since $\widetilde{F}(M\tensor_A C)\cong M\tensor_A \widetilde{F}(C)$ in $\VV_B^D$. Commutativity of the right triangle is simply the counit axiom for the $C$-comodule structure on $M$. Axioms \eqref{eq:pentagon} and \eqref{eq:counit} hold automatically, because they are equivalent to saying that \eqref{eq:comodule} defines a $D$-comodule structure on $A\tensor_A X$ (cf. Remark \ref{rmk:interpret}), and we know a priori that $\delta_A$ defines a $D$-comodule structure on $A\tensor_A X$.
\end{proof}

\begin{remark}
Not every adjunction between $\VV_A^C$ and $\VV_B^D$ is governed by a braided bimodule. For instance, this is usually not the case for adjunctions arising from twisting cochains.
\end{remark}

\subsection{Cotensor products}
The right adjoint $T^{*}$ in the adjunction governed by a braided bimodule $(X,T)$ is difficult to describe in general. However, we will show that under appropriate conditions on the underlying (bi)modules ${}_AX_B$  and ${}_{A}C$, it is possible to express $T^*$ as a cotensor product.

\begin{definition}\label{defn:a-flat}
We call a left $A$-module $N$ \emph{flat} if the functor $-\tensor_A N\colon \VV_A\to \VV$ preserves coreflexive equalizers.
\end{definition}

\begin{remark}
If $\VV$ is an abelian category, then it is easy to show that $-\tensor_A N$ preserves coreflexive equalizers if and only if $N$ is flat in the usual sense that $-\tensor_A N$ preserves monomorphisms.

Note that the notions of flatness and homotopy flatness (Definition \ref{defn:special-modules}) are different in general. For instance, if the weak equivalences in $\VV$ are the isomorphisms, then every module is homotopy flat but not necessarily flat. 
\end{remark}

\begin{definition}\label{defn:flat-coring}
We will call a coring $(A,C)$ \emph{flat} if $C$ is flat as a left $A$-module, i.e., if $-\tensor_A C\colon \VV_A\to \VV$ preserves coreflexive equalizers.
\end{definition}

The next proposition follows directly from the observations in Remark \ref{rmk:coref-eq}.
\begin{proposition}
If $(A,C)$ is a flat coring, then the forgetful functor $\UU_A\colon \VV_A^C\to \VV_A$ creates coreflexive equalizers. 
\end{proposition}

\begin{definition}
Suppose that $\VV$ admits coreflexive equalizers. Let $(A,C)$ be a coring in $\VV$, let $M$ be a right and $N$ a left $(A,C)$-comodule. The \emph{cotensor product} $M\cotensor_C N$ is defined as the coreflexive equalizer in $\VV$:
$$
\equalizer{M\cotensor_C N}{M\tensor_A N}{M\tensor_A C \tensor_A N.}{\delta_M\tensor 1}{1\tensor \delta_N}
$$
\end{definition}

\begin{proposition} \label{prop:R_X dualizable}
Let $\VV$ be a closed symmetric monoidal category admitting all reflexive coequalizers and coreflexive equalizers. Let $(A,C)$ be a flat coring in $\VV$, and let $(X,T)\colon (A,C)\rightarrow (B,D)$ be a braided bimodule. 

If the underlying bimodule ${}_AX_B$ admits a strict right dual $X^\vee$, then $X^\vee\tensor_A C$ is a left $(B,D)$-comodule in $\VV_A^C$, and the right adjoint of the adjunction governed by $(X,T)$ is isomorphic to the cotensor product functor $-\cotensor_{D}(X^\vee \tensor_A C)$, i.e., there is an adjunction
$$
\hugeadjunction{\VV_A^C}{\VV_B^D}{T_{*}}{-\cotensor_{D} (X^\vee\tensor_A C)}.
$$
\end{proposition}

\begin{proof}
The left $D$-comodule structure on $X^\vee\tensor_A C$ is defined by the following composite:
\begin{align*}
X^\vee \tensor_A C & \xrightarrow{1\otimes \Delta} X^\vee \tensor_A C\tensor_A C \\
& \xrightarrow{1\otimes u\otimes 1} X^\vee \tensor_A C \tensor_A X\tensor_B X^\vee \tensor_A C \\
& \xrightarrow {1\otimes T\otimes 1} X^\vee \tensor_A X\tensor_B D\tensor_B X^\vee \tensor_A C \\
& \xrightarrow {e\otimes 1} D\tensor_B X^\vee \tensor_A C.
\end{align*}
Here, $\Delta$ is the comultiplication on $C$, the map $u$ is the coevaluation, and $e$ is the evaluation. We leave it to the reader to verify that the axioms for a comodule are satisfied. It follows from the natural isomorphism $N\tensor_B X^\vee \cong \Map_B(X,N)$, which holds because $X$ is dualizable, that the coreflexive equalizer diagram \eqref{eq:R_X} defining $T^{*}(N)$ may be identified with the equalizer diagram defining the cotensor product $N\cotensor_D (X^\vee\tensor_A C)$. Note that there is a subtlety in that the coreflexive equalizer \eqref{eq:R_X} should be calculated in  $\VV_A^C$, whereas the coreflexive equalizer defining the cotensor product should be calculated in $\VV_A$. Since we assume that the coring $(A,C)$ is flat, the forgetful functor $\UU_A\colon \VV_A^C\to \VV_A$ creates coreflexive equalizers, so we may identify the two.
\end{proof}

Important special cases of braided bimodules for which Proposition \ref{prop:R_X dualizable} applies are the braided bimodules associated to morphisms of corings. Indeed, let $(\varphi,f)\colon (A,C)\rightarrow (B,D)$ be a morphism of corings and let $(B,T_{\varphi,f})$ be the associated braided bimodule (see Example \ref{ex:braided bimodules 2}). The underlying bimodule $X={}_A B_B$ is dualizable, with right dual $X^\vee = {}_B B_A$. If the coring $(A,C)$ is flat, it follows that the adjunction governed by the morphism $(\varphi,f)$ can be written as
$$
\bigadjunction{\VV_A^C}{\VV_B^D}{-\tensor_A B}{-\cotensor_D (B\tensor_A C)}.
$$
Specializing further, if $A=B$, $\varphi\colon A\rightarrow B$ is the identity map, and $f\colon C\rightarrow D$ is a morphism of $A$-corings, we recover the familiar change of corings adjunction
$$
\bigadjunction{\VV_A^C}{\VV_A^D}{f_*}{-\cotensor_D C}.
$$

\subsection{The canonical coring} \label{subsec:canonical coring}
In this section we introduce the \emph{canonical coring} $X_*(C)$ associated to a coring $(A,C)$ and a right dualizable bimodule ${}_AX_B$. The canonical coring generalizes the descent coring associated to a morphism of algebras, and it will be useful for our analysis of Quillen equivalences between comodule categories in Section \ref{sec:m-t}.

\begin{proposition} \label{prop:universal coring}
Let $A$ and $B$ be algebras in $\VV$, and let ${}_AX_B$ be a right dualizable bimodule. For every $A$-coring $C$, there is a $B$-coring $X_*(C)$ and a braided bimodule $(X,T_C^{\mathrm{univ}})\colon (A,C) \rightarrow \big(B,X_*(C)\big)$ satisfying the following universal property: for every coring $(B',D)$ and every braided bimodule $(Z,T)\colon (A,C)\rightarrow (B',D)$ with $Z=X\tensor_B Y$ for some bimodule ${}_BY_{B'}$, there is a unique braided bimodule $(Y,S)$ such that the following diagram in $\CORING_\VV$ commutes.
$$
\xymatrix{(A,C) \ar[rr]^-{(X,T_{C}^{\mathrm{univ}})} \ar[drr]_-{(X\tensor_B Y,T)\;\;} && (B,X_*(C)) \ar@{-->}[d]^-{(Y,S)} \\
&& (B',D)}
$$
In particular, for every braiding $T\colon C\tensor_A X\rightarrow X\tensor_B D$, there is a unique morphism of $B$-corings 
$$g_{T}\colon X_{*}(C) \rightarrow D$$ 
making the following diagram in $\CORING_\VV$ commute:
$$
\xymatrix{(A,C) \ar[rr]^-{(X,T_{C}^{\mathrm{univ}})} \ar[drr]_-{(X,T)} && (B,X_{*}(C)) \ar@{-->}[d]^-{(B, T_{1,g_{T}})} \\
&& (B,D).}
$$
\end{proposition}

\begin{remark}
Proposition \ref{prop:universal coring} may be formulated succinctly by saying that $(X,T_C^{\mathrm{univ}})\colon (A,C)\rightarrow (B,X_*(C))$ is a co-cartesian morphism over $X\colon A\rightarrow B$, under the forgetful functor $\CORING_\VV\rightarrow \ALG_\VV$.  In other words, the pullback of the forgetful functor $\CORING_\VV\rightarrow \ALG_\VV$ along the subcategory of $\ALG_\VV$ consisting of right dualizable morphisms is a co-cartesian fibration. In particular, since the underlying bimodule of the braided bimodule associated to a morphism of corings is always right dualizable, the category $\Coring_\VV$ is cofibered over $\Alg_\VV$.
\end{remark}

\begin{proof}
Let $u\colon A\rightarrow X\tensor_B X^\vee$ and $e\colon X^\vee\tensor_A X\rightarrow B$ denote the coevaluation and evaluation maps, and let $\Delta\colon C\to C\otimes_{A}C$ and $\epsilon: C\to A$ denote the comultiplication and counit of the $A$-coring $C$. We only define $X_*(C)$ and the structure maps, leaving the straightforward verification of their properties to the reader.

Define $X_*(C)$ to be the $B$-bimodule
$$X_*(C) = X^\vee\tensor_A C\tensor_A X,$$
and define the comultiplication as the composite 
$$
\xymatrix{X^\vee \tensor_A C\tensor_A X \ar[rr]^-{1\tensor \Delta\tensor 1} && X^\vee \tensor_A C\tensor_A C\tensor_A X \ar[d]^-{1\tensor 1\tensor u \tensor 1\tensor 1} \\
&&\big(X^\vee \tensor_A C\tensor_A X\big)\tensor_B \big(X^\vee \tensor_A C\tensor_A X\big).}
$$
The counit is defined to be the composite
$$
\xymatrix{X^\vee \tensor_A C\tensor_A X \ar[rr]^-{1\tensor \epsilon\tensor 1} && X^\vee \tensor_A X \ar[r]^-{e} & B.}
$$
The universal braiding is defined to be
$$T_{C}^{\mathrm{univ}}=u\tensor 1 \colon C\tensor_A X \rightarrow X\tensor_B \big( X^\vee \tensor_AC\tensor_A X \big).$$

Finally, given a coring $(B',D)$ and a braided bimodule 
$$(X\tensor_B Y,T)\colon (A,C)\rightarrow (R,D),$$ 
we define the braided bimodule $(Y,S)$ by letting $S$ be the composite
$$
\xymatrix{X^\vee\tensor_A C\tensor_A X\tensor_B Y \ar[r]^-{1\tensor T} & X^\vee \tensor_A X\tensor_B Y\tensor_R D \ar[r]^-{e\tensor 1\tensor 1} & Y\tensor_R D.}
$$
In the special case of a $B$-coring $D$ and a braiding $T\colon C\tensor_A X \rightarrow X\tensor_B D$, the morphism $g_{T}$ is  the composite
$$
\xymatrix{X^\vee\tensor_A C\tensor_A X \ar[r]^-{1\tensor T} & X^\vee \tensor_AX\tensor_B D \ar[r]^-{e\tensor 1} & D.}
$$
\end{proof}

\begin{definition}\label{defn:canonical}
We will refer to the $B$-coring $X_*(C)$ introduced in Proposition \ref{prop:universal coring} as the \emph{canonical coring} associated to $X$ and $(A,C)$.

Furthermore, we define the \emph{canonical adjunction associated to $X$ and $C$} to be the adjunction governed by the universal braided bimodule $(X,T_C^{\mathrm{univ}})$,
\begin{equation} \label{eq:descent adjunction}
\hugeadjunction{\VV_A^C}{\VV_B^{X_*(C)}.}{(T_C^{\mathrm{univ}})_*}{(T_C^{\mathrm{univ}})^*}
\end{equation}
\end{definition}

\begin{remark}The canonical coring $X_*(C)$ generalizes the descent coring associated to a morphism of algebras $\varphi\colon A\rightarrow B$. Indeed, if  $C$ is the trivial $A$-coring $A$, and $X$ is  the bimodule ${}_A B_B$, then the canonical coring $B_*(A)$ is isomorphic to the descent coring  $\Desc(\varphi)$ associated to $\varphi$ \cite{mesablishvili}, \cite{hess:descent}, which has underlying $B$-bimodule $B\otimes_{A}B$. Moreover, the canonical adjunction is the adjunction
$$
\bigadjunction{\VV_A}{\VV_B^{\Desc(\vp)}}{\Can_\varphi}{\Prim_\varphi}
$$
familiar from descent theory.
Note that $B$ itself is an object in $\VV_B^{\Desc(\varphi)}$, where the right $\Desc(\vp)$-coaction is given by
  $$B\cong A\otimes_{A}B \xrightarrow {\vp\otimes 1} B\otimes_{A}B\cong B\otimes_{B}(B\otimes_{A}B).$$

\end{remark}

\begin{notation}\label{notn:can-prim}
Motivated by the remark above, we write henceforth
$$\Can_{X}=(T_{C}^{\mathrm{univ}})_{*}: \VV_A^C\to \VV_B^{X_*(C)}$$
and
$$\Prim_{X}=(T_{C}^{\mathrm{univ}})^{*}:\VV_B^{X_*(C)}\to \VV_A^C$$
for every universal braided bimodule $(X, T_{C}^{\mathrm{univ}})$.
\end{notation}

\begin{example}\label{ex:canonical-coringmap}
Let $(\vp,f )\colon (A,C) \to (B,D)$ be a morphism of corings with associated braided bimodule $(B, T_{\vp, f})$ (Example \ref{ex:braided bimodules 2}). A straightforward calculation shows that
$$g_{T_{\vp, f}}=f\colon B\tensor_A C\tensor_A B \to D.$$
\end{example}

\subsection{Dualizable braided bimodules}
In this section we analyze the notion of dualizability for braided bimodules. This turns out to require more than simply the dualizability of the underlying bimodule. We refer the reader to Appendix \ref{app:dualizable} for a brief overview of dualizability and adjunctions in bicategories. Throughout this section $\VV$ is a symmetric monoidal category.

Recall that $\CORING_\VV$ denotes the bicategory of corings and braided bimodules (see Definition \ref{def:coring bicat}). Within the general categorical framework of Appendix \ref{app:dualizable}, there is an appropriate definition of dual braided bimodules. 

\begin{definition} A braided bimodule $(X,T)\colon (A,C)\to (B,D)$ is \emph{left dual} to $(X^\vee,T^\vee)\colon (B,D)\to (A,C)$ if $(X,T)$ is left adjoint to $(X^\vee,T^\vee)$ in the bicategory $\CORING_{\VV}$ (cf.~Definition \ref{defn:dual}). 
\end{definition}
Let us make this more explicit.
\begin{proposition}
A braided bimodule $(X,T)\colon (A,C)\to (B,D)$ is left dual to $(X^\vee,T^\vee)\colon (B,D)\to (A,C)$ if and only if the underlying bimodules are strictly dual (cf.~Example \ref{ex:bimod}), via a coevaluation $u\colon A \to X\otimes _{B}X^\vee$ and an evaluation $e\colon X^\vee\otimes_{A}X\to B$, and the diagrams
\begin{equation} \label{eq:eta}
\xymatrix{
C\tensor_A A \ar[d]^-{1\tensor u} \ar[rr]^-\cong && A\tensor_A C \ar[d]^-{u\tensor 1} \\
C\tensor_AX\tensor_B X^\vee \ar[dr]_-{T\tensor 1} && X\tensor_B X^\vee \tensor_A C \\
& X\tensor_B D \tensor_B X^\vee \ar[ur]_-{1\tensor T^\vee}}
\end{equation}
\begin{equation} \label{eq:epsilon}
\xymatrix{ & X^\vee\tensor_AC\tensor_A X \ar[dr]^-{1\tensor T} \\
D\tensor_B X^\vee \tensor_A X \ar[d]^-{1\tensor e} \ar[ur]^-{T^\vee\tensor 1} && X^\vee\tensor_A X\tensor_B D \ar[d]^-{e\tensor 1} \\
D\tensor_B B \ar[rr]^-\cong && B\tensor_B D}
\end{equation}
commute.
\end{proposition}

\begin{proof}
The proof simply amounts to unwinding the definition of an adjunction in a bicategory. The commutativity of the diagrams \eqref{eq:eta} and \eqref{eq:epsilon} corresponds exactly to requirement that the coevaluation $u\colon A\rightarrow X\tensor_B X^\vee$ and evaluation $e\colon X^\vee\tensor_A X\rightarrow B$ be morphisms of braided bimodules.
\end{proof}

\begin{remark} Let $\VV$ be a closed, symmetric monoidal category. Let $\VCAT$ denote the bicategory (in fact, $2$-category) of $\VV$-categories, $\VV$-functors, and $\VV$-natural transformations. There is a bifunctor
$$\VV\colon \CORING_\VV\rightarrow \VCAT$$
that sends a coring $(A,C)$ to the $\VV$-category $\VV_A^C$, a braided bimodule $(X,T)$ to the associated $\VV$-functor $T_{*}\colon \VV_A^C\rightarrow \VV_B^D$, and a morphism $f\colon (X,T)\rightarrow (X',T')$ of braided bimodules to the obvious natural transformation that sends an object $(M,\delta)$ in $\VV_{A}^{C}$ to the morphism
$$1\otimes f \colon M\otimes_{A}X \to M\otimes_{A'}X',$$
where $(X,T), (X',T') \colon (A,C)\to (B,D)$.  Since bifunctors preserve adjunctions, it follows that if $(X,T)\colon (A,C) \to (B,D)$ is a dualizable braided bimodule with dual $(X^{\vee},T^{\vee}) \colon (B,D) \to (A,C)$, then 
$$\bigadjunction{\VV_A^C}{\VV_B^D}{T_{*}}{(T^{\vee})_{*}}$$
is a $\VV$-adjunction, whence there exists a natural isomorphism
$$(T^{\vee})_{*}\cong T^{*}\colon \VV_B^D\to \VV_A^C,$$
by uniqueness up to isomorphism of right adjoints.
\end{remark}

\begin{remark}
The bimodule ${}_AB_B$ arising from a morphism of algebras $\varphi\colon A\rightarrow B$ is dualizable. Unfortunately, it is not true in general that the braided bimodule associated to a morphism of corings $(\varphi,f)\colon (A,C)\rightarrow (B,D)$ is dualizable. Indeed, let $\varphi\colon A\rightarrow B$ be a morphism of algebras, and consider the induced morphism of corings $(A,A)\rightarrow (B,B\tensor_A B)$. The induced adjunction is the descent adjunction
$$\bigadjunction{\VV_A}{\VV_B^{B\tensor_A B}}{\Can_\varphi}{\Prim_\varphi}.$$
The braided bimodule that governs this adjunction is $X = {}_AB_B$, where $B$ is viewed as a left $A$-module via the morphism $\varphi$. The braiding is given by the canonical right $B\tensor_A B$-comodule structure on $B$ (cf. Remark \ref{rem:C=A}), $T\colon A\tensor_A B \rightarrow B\tensor_B (B\tensor_A B)$. The underlying bimodule is dualizable with right dual $X = {}_BB_A$. However, the braided bimodule $(B,T)$ is dualizable if and only if $\Prim_\varphi$ is isomorphic to the forgetful functor, which rarely happens.
\end{remark}

\section{Homotopical Morita theory for corings} \label{sec:m-t}
In this section we elaborate a homotopical version of Morita-Takeuchi theory for corings (cf.~\cite[\S23]{brzezinski-wisbauer}), providing conditions under which two $\VV$-model categories of comodules over corings are Quillen equivalent. In particular, we provide criteria in terms of homotopic descent under which a morphism of corings induces a Quillen equivalence between the associated comodule categories.

\begin{hypothesis} \label{hypo:m-t}
Throughout this section, $\VV$ denotes a symmetric monoidal model category (see Appendix \ref{app:enriched}). Moreover, we assume that for every coring $(A,C)$, the category $\VV_A^C$ admits the model structure left-induced from a model structure $\VV_{A}$, via the forgetful adjunction
\begin{equation} \label{eq:adjunction}
\bigadjunction{\VV_A^C}{\VV_A}{\UU_A}{-\tensor_A C}.
\end{equation}
In other words, a morphism in $\VV_A^C$ is a weak equivalence (cofibration) if and only if it is a weak equivalence (cofibration) when regarded as a morphism in $\VV_A$.
\end{hypothesis}

Currently, this hypothesis is known to hold in the following examples.
\begin{itemize}
\item For $\VV=\mathsf{sSet}$, the category of simplicial sets, where the monoidal product is the cartesian product, and for every simplicial monoid $A$,  \cite[Theorem 2.2.3]{hkrs} implies that $\mathsf{sSet}_{A}$ admits an injective model structure, for which the weak equivalences are the weak homotopy equivalences, and the cofibrations are the injections.  By an argument analogous to that in \cite[Theorem 5.0.3]{hkrs}, the left-induced model structure on $(\mathsf{sSet})_{A}^{C}$ exists for all $A$-corings $C$.
\item For $\VV=\mathsf{sSet}_{*}$, the category of pointed simplicial sets, where the monoidal product is the smash product, and for every simplicial monoid $A$,  we again apply \cite[Theorem 2.2.3]{hkrs} to conclude that $(\mathsf{sSet}_{*})_{A}$ admits an injective model structure, for which the weak equivalences are the based weak homotopy equivalences, and the cofibrations are the based injections.  Again by an argument analogous to that in \cite[Theorem 5.0.3]{hkrs}, the left-induced model structure on $(\mathsf{sSet}_{*})_{A}^{C}$ exists for all $A$-corings $C$.
\item For $\VV = \Sp^\Sigma$, the category of symmetric spectra, and for every symmetric ring spectrum $A$, where $\Sp_A^\Sigma$  has the injective model structure, for which the weak equivalences are the stable equivalences and the cofibrations are the injections, the left-induced model structure on $(\Sp^\Sigma)_{A}^{C}$ exists for all $A$-corings $C$  (see \cite[Theorem 5.0.3]{hkrs}).

\item For $\VV = \Ch_R$, the category of chain complexes over a commutative ring $R$, and a dg algebra $A$, we can give $(\Ch_R)_A$ the injective model structure, where the weak equivalences are the quasi-isomorphisms and the cofibrations are the levelwise injections (see \cite[Theorem 6.6.3]{hkrs}).

\item For $\VV= \Ch_R$, we could alternatively give $(\Ch_R)_A$ the $r$-module structure, where the weak equivalences are the $R$-chain homotopy equivalences (see \cite[Theorem 6.6.3]{hkrs})
\end{itemize}

Presumably, more examples can be added to this list.

\begin{remark}
In all examples we consider in this paper, the model structure on $\VV_A$ satisfies Hypothesis \ref{hypo:module}, i.e., the weak equivalences in $\VV_A$, and hence in $\VV_A^C$, are created in $\VV$ via the forgetful functor. However, in the discussion to follow, it is not necessary that the weak equivalences in $\VV_A$ are created in $\VV$ --- any model structure on $\VV_A$ will do.
\end{remark}


\begin{remark}
Under the hypothesis, $\VV$, $\VV_{A}$, and $\VV_A^C$ are model categories, so they are in particular complete and cocomplete and thus admit all reflexive coequalizers and coreflexive equalizers.
\end{remark}

\begin{remark}
By definition of the $\VV$-tensor structure on $\VV_A^C$ (see Proposition \ref{prop:comodule v-category}), the left adjoint $\UU_A$ is a tensor functor, so it follows from Proposition \ref{prop:V-adjunction} that the adjunction \eqref{eq:adjunction} is $\VV$-structured. By Proposition \ref{prop:v-induced}, it follows that the model structure on $\VV_A^C$ is $\VV$-structured, whenever the model structure on $\VV_A$ is.
\end{remark}

\subsection{Towards Quillen equivalences of comodule categories}

We begin our study of the homotopy theory of comodules over corings by providing conditions under which an adjunction governed by a braided bimodule is a Quillen adjunction.

\begin{proposition} \label{prop:Quillen-comod} Let $\VV$ be a symmetric monoidal model category satisfying Hypothesis \ref{hypo:m-t}.
Let $(X,T)\colon (A,C) \to (B,D)$ be a braided bimodule. If
$$\bigadjunction{\VV_A}{\VV_B}{-\tensor_A X}{\Map_B(X,-)}$$
is a Quillen adjunction, then so is
$$\bigadjunction{\VV_A^C}{\VV_B^D}{T_{*}}{T^{*}}.$$
The converse holds if $(A,C)$ is coaugmented.
\end{proposition}

\begin{proof}
This follows readily from the fact that the model structure on $\VV_A^C$ is left induced via the forgetful functor $\UU_A\colon \VV_A^C\to \VV_A$.
\end{proof}

The next result, which is a sort of dual to Proposition \ref{prop:resext}, is a first step towards understanding Quillen equivalences of comodule categories. 

\begin{proposition}\label{prop:we-QE} Let $\VV$ be a symmetric monoidal model category satisfying Hypothesis \ref{hypo:m-t}.
Let $f\colon C\to D$ be a morphism of $A$-corings. The change-of-corings adjunction,
$$
\bigadjunction{\VV_A^C}{\VV_A^D,}{f_*}{f^{*}}
$$
is a Quillen equivalence if and only if the counit of the adjunction,
$$\epsilon_{M}:f_{*}f^{*}(M) \to M,$$
is a weak equivalence for all fibrant right $D$-comodules $M$.
If $A$ is fibrant in $\VV_A$, and the change-of-corings adjunction is a Quillen equivalence, then $f$ is a weak equivalence.
\end{proposition}

\begin{remark}
If the coring $(A,C)$ is flat, then $f^{*}(M)=M\cotensor_{D}C$ by Proposition \ref{prop:R_X dualizable}. In this case, if $f$ is a weak equivalence, and the functor $M\cotensor_{D}-: {}_{A}^{D}\VV \to \VV$ preserves weak equivalences for all fibrant right $D$-comodules $M$, then the adjunction above is a Quillen equivalence.
\end{remark}

\begin{proof}
The functor $f_*$ preserves and reflects weak equivalences, because these are created in the underlying category $\VV_{A}$, and $f_*$ does not change the underlying $A$-module. It follows that the adjunction is a Quillen equivalence if and only if the counit of the adjunction
is a weak equivalence for all fibrant objects $M$ in $\VV_A^D$. 

If $A$ is fibrant in $\VV_{A}$, it follows that $D$ is fibrant as an object of $\VV_A^D$ because it is the image of $A$ under the right Quillen functor $-\tensor_A D\colon \VV_A \rightarrow \VV_A^D$. Thus,  if the counit $\epsilon_M$ is a weak equivalence for all fibrant $M$, then $f$ is necessarily a weak equivalence, because the component of the counit at $D$ is exactly $f$ (cf.~Example \ref{ex:braided bimodules 3}). 
\end{proof}

Since the condition on $f$ in Proposition \ref{prop:we-QE} recurs throughout the rest of this article, we give it a name, dual to that for the condition that arose in the module case (Proposition \ref{prop:resext}). In Section \ref{sec:examples} we provide concrete examples of chain maps satisfying this condition.

\begin{definition} \label{defn:copure}
Let $\VV$ be a symmetric monoidal model category satisfying Hypothesis \ref{hypo:m-t}. We say that a weak equivalence $f\colon C\to D$ of $A$-corings is \emph{copure} if
$$\epsilon_{M}:f_{*}f^{*}(M) \to M$$ 
is a weak equivalence for all fibrant right $D$-comodules $M$.
\end{definition}

\subsection{Homotopic descent over corings}
Our description of the canonical coring associated to a dualizable $A$-$B$-bimodule (Definition \ref{defn:canonical}) hints at an interesting generalization of the usual notion of homotopic descent \cite{hess:descent}, which turns out to be important for our discussion of Quillen equivalences of comodule categories. Recall Notation \ref{notn:can-prim}.

\begin{definition} \label{defn:homotopic descent}
Let $(A,C)$ be a coring in $\VV$ and $B$ an algebra in $\VV$. A strictly dualizable $A$-$B$-bimodule $X$ \emph{satisfies effective homotopic descent with respect to $C$} if the canonical adjunction
\begin{equation}\label{eqn:canonical-adj}
\bigadjunction{\VV_A^C}{\VV_B^{X_*(C)},}{\Can_{X}}{\Prim_{X}}
\end{equation}
is a Quillen equivalence.
\end{definition}

\begin{remark}\label{rmk:eff-htpic-desc}  
This extends the definition of effective homotopic Grothendieck descent \cite{hess:descent}. Indeed, a morphism of algebras $\varphi\colon A\to B$ satisfies effective homotopic descent, in the sense of \cite[\S6]{hess:descent}, if and only if the bimodule $X= {}_AB_B$ satisfies effective homotopic descent with respect to the trivial coring $A$, in the sense of Definition \ref{defn:homotopic descent}.
\end{remark}

When specialized to $\VV = \Ab$ (the category of abelian groups with weak equivalences = isomomorphisms) the following theorem recovers Grothendieck's classical theorem on faithfully flat descent for morphism of commutative rings. Recall the definitions of the special classes of modules in Definition \ref{defn:special-modules} and the definition of a flat coring from Definition \ref{defn:flat-coring}.

\begin{theorem} \label{thm:homotopic descent}
Let $\VV$ be a symmetric monoidal model category satisfying Hypothesis \ref{hypo:m-t}. Let $(A,C)$ be a flat coring, let $B$ be an algebra, and let ${}_AX_B$ be a right dualizable bimodule. If $X$ is homotopy faithfully flat as a left $A$-module, then $X$ satisfies effective homotopic descent with respect to $(A,C)$.
\end{theorem}

\begin{proof}
The underlying right $B$-module of $\Can_X(M)$ is simply $M\otimes_{A}X$. Since $X$ is homotopy faithfully flat as a left $A$-module, and since weak equivalences are detected in $\VV_B$, it follows that the functor $\Can_X \colon \VV_A^C \rightarrow \VV_B^{X_*(C)}$ preserves and reflects weak equivalences. Consequently, \eqref{eqn:canonical-adj} is a Quillen equivalence if and only if the counit of the adjunction is a weak equivalence for every fibrant object $M$. 

Since $(A,C)$ is flat, and $X$ is right dualizable, the right adjoint in \eqref{eqn:canonical-adj} may be expressed as a cotensor product (Proposition \ref{prop:R_X dualizable}). It follows that the counit of the adjunction may be identified with the map
$$\big(M\cotensor_{X_*(C)} (X^\vee\tensor_A C)\big) \tensor_A X \longrightarrow M \cotensor_{X_*(C)} X_*(C) \cong M$$
induced by the universal property of the cotensor product. Since $-\tensor_A X$  preserves finite limits up to homotopy, it follows that the counit is a weak equivalence.
\end{proof}

The corollary below is an important special case of Theorem \ref{thm:homotopic descent}.

\begin{corollary}
Let $\varphi\colon A\rightarrow B$ be a morphism of algebras in $\VV$. If $B$ is homotopy faithfully flat as a left $A$-module, then $\varphi$ satisfies effective homotopic descent.
\end{corollary}

\begin{remark} \label{rem:classical}
The classical theorem is recovered by taking $\VV$ to be the category of abelian groups, with isomorphisms as weak equivalences, $C$ to be the trivial $A$-coring $A$, and $X$ to be the bimodule ${}_AB_B$, where $B$ is viewed as a left $A$-module via the morphism of algebras $\varphi\colon A\rightarrow B$. Note that in this setting a module is homotopy faithfully flat if and only if it is faithfully flat in the ordinary ring theoretic sense.
\end{remark}

\begin{remark}
Homotopical faithful flatness is not necessary for homotopic effective descent. In fact, the analogous condition is already not necessary for ordinary effective descent for commutative rings. A morphism of commutative rings satisfies effective descent if and only it is \emph{pure}, see e.g.~\cite[Exercise 4.8.12]{borceux} or \cite{janelidze-tholen}. Pure morphisms are necessarily faithful, but not necessarily flat; for example, $\ZZ\to \ZZ\times \ZZ/n\ZZ$, $m\mapsto (m,[m])$ is a pure but non-flat homomorphism of $\ZZ$-modules.
\end{remark}

\subsection{The homotopical Morita-Takeuchi theorem}
Assembling  our results on Quillen equivalences induced by copure coring maps (Proposition \ref{prop:we-QE}) and on effective homotopic descent over corings (Theorem \ref{thm:homotopic descent}), we can now approach the question formulated in the introduction and provide conditions under which two categories of comodules over corings are Quillen equivalent. Recall that if ${}_{A}X_{B}$ is strictly right dualizable, then every braided module $(X,T) \colon (A,C) \to (B,D)$ determines a morphism of corings $g_{T}\colon \big(B, X_{*}(C)\big) \to (B,D)$ (see Proposition \ref{prop:universal coring}).

\begin{theorem} \label{thm:bb}
Assume Hypothesis \ref{hypo:m-t}.

Let $(X,T)\colon (A,C)\to (B,D)$ be a braided bimodule such that $-\tensor_A X\colon \VV_A \to \VV_B$ is a left Quillen functor and $X_B$ is strictly dualizable. Let $g_{T}\colon X_*(C) \to D$ denote the associated morphism of $B$-corings.

If $X$ satisfies effective homotopic descent with respect to $C$, and $g_{T}\colon X_*(C) \to D$ is a copure weak equivalence of corings, then the Quillen adjunction governed by $(X,T)$,
\begin{equation} \label{eq:coring-adjunction}
\bigadjunction{\VV_A^C}{\VV_B^D}{T_{*}}{T^{*}},
\end{equation}
is a Quillen equivalence.

Conversely, if $B$ is fibrant in $\VV_B$, the coring $(A,C)$ is flat, and ${}_AX$ is strongly homotopy flat, then \eqref{eq:coring-adjunction} is a Quillen equivalence only if $X$ satisfies effective homotopic descent with respect to $C$, and $g_{T}\colon X_*(C) \to D$ is a copure weak equivalence of corings.
\end{theorem}

\begin{proof}
The adjunction (\ref{eq:coring-adjunction}) is a Quillen adjunction by Proposition \ref{prop:Quillen-comod}. Since $X$ is right dualizable, it follows from Proposition \ref{prop:universal coring} that the adjunction \eqref{eq:coring-adjunction} factors as a generalized descent adjunction followed by a change-of-corings adjunction,
\begin{equation} \label{eq:braided factorization}
\xymatrix{ \VV_A^C \ar@<1ex>[rrr]^-{\Can _{X}} &&& \VV_B^{X_*(C)} \ar@<1ex>[lll]^-{\Prim_{X}} \ar@<1ex>[rrr]^-{(g_{T})_*} &&& \VV_B^D. \ar@<1ex>[lll]^-{(g_{T})^{*}}}
\end{equation}
If $X$ satisfies effective homotopic descent with respect to $C$, then the first adjunction in the factorization \eqref{eq:braided factorization} is a Quillen equivalence by definition (see Definition \ref{defn:homotopic descent}). If $g_{T}\colon X_*(C) \rightarrow D$ is a copure weak equivalence, then the second adjunction in \eqref{eq:braided factorization} is a Quillen equivalence by Proposition \ref{prop:we-QE}. It follows that \eqref{eq:coring-adjunction} is a Quillen equivalence.

Conversely, suppose that $B$ is fibrant in $\VV_B$, $(A,C)$ is flat, ${}_AX$ is strongly homotopy flat, and the adjunction \eqref{eq:coring-adjunction} is a Quillen equivalence. Proposition \ref{prop:R_X dualizable} implies that the component of the counit of the adjunction \eqref{eq:coring-adjunction} at $M\in \VV_B^D$ may be identified with the composite
$$\big( M\cotensor_D (X^\vee \tensor_A C) \big) \tensor_A X\to M\cotensor_D (X^\vee \tensor_A C \tensor_A X) \xrightarrow{M\cotensor_D g_T} M\cotensor_D D \cong M,$$
which represents the homotopy counit when $M$ fibrant (since ${}_AX$ is homotopy flat) and is therefore a weak equivalence, as \eqref{eq:coring-adjunction} is a Quillen equivalence. Moreover, since ${}_AX$ is strongly homotopy flat, the first map in the composite is a weak equivalence.   Hence, $M\cotensor_D g_T\colon  M\cotensor_D X_*(C) \to M\cotensor_D D$ is a weak equivalence for all fibrant $M\in \VV_B^D$ and in particular for $M=D=B\otimes _{B}D$, which is fibrant since $B$ is fibrant in $\VV_{B}$.  In other words,  $g_T\colon X_*(C)\to D$ is a copure weak equivalence. It follows from Proposition \ref{prop:we-QE} that the second adjunction in the factorization \eqref{eq:coring-adjunction} is a Quillen equivalence. By the 2-out-of-3 property for Quillen equivalences, the first adjunction must be a Quillen equivalence as well, i.e., $X$ satisfies effective homotopic descent with respect to $C$.
\end{proof}

The special case when the adjunction is induced by a morphism of corings is worth singling out. Recall Examples \ref{ex:braided bimodules 2} and \ref{ex:canonical-coringmap}.

\begin{corollary} \label{thm:we=QE for corings}
Assume Hypothesis \ref{hypo:m-t}.

Let $(\varphi,f)\colon (A,C)\rightarrow (B,D)$ be a morphism of corings in $\VV$ and suppose that $-\tensor_A B\colon \VV_A \to \VV_B$ is a left Quillen functor.

If the morphism $\varphi\colon A\rightarrow B$ satisfies effective homotopic descent with respect to $C$, and the morphism of $B$-corings $f\colon B_*(C) \rightarrow D$ is a copure weak equivalence, then the adjunction governed by $(\varphi,f)$,
$$\bigadjunction{\VV_A^C}{\VV_B^D}{(T_{\vp, f})_{*}}{(T_{\vp, f})^{*}},$$
is a Quillen equivalence. 

Conversely, if $(A,C)$ is flat, and $B$ is strongly homotopy flat as a left $A$-module and fibrant as a right $B$-module, then the adjunction governed by $(\varphi,f)$ is a Quillen equivalence only if $\varphi$ satisfies effective homotopic descent with respect to $C$, and $f\colon B_*(C) \rightarrow D$ is a copure weak equivalence.
\end{corollary}

\begin{remark}
When specialized to Hopf algebroids in the category of graded modules over a commutative ring (with weak equivalences being the isomorphisms), Corollary \ref{thm:we=QE for corings} recovers \cite[Theorem 6.2]{hovey-strickland} and \cite[Theorem D]{hovey}.
\end{remark}

\section{The case of unbounded chain complexes} \label{sec:examples}
We now apply the theory of the previous sections to the category $\Ch_{R}$ of unbounded chain complexes over an arbitrary commutative ring $R$. 
Barthel, May, and Riehl showed that the category $\Ch_{R}$ admits a model structure where the weak equivalences are the chain homotopy equivalences, the cofibrations are the degreewise-split injections, and the fibrations are the degreewise-split surjections \cite[Theorem 1.15]{barthel-may-riehl}, which we call the \emph{Hurewicz model structure}. This model category structure is closed monoidal with respect to the usual graded tensor product of chain complexes, where the internal hom is the usual unbounded hom-complex.

The following homotopical version of the well known Five Lemma plays an important role throughout the analysis in the rest of this section.

\begin{lemma}[The Homotopy Five Lemma]\label{lem:h5} Let
$$\xymatrix{0\ar[r]&M'\ar [r]\ar[d]_{f'}&M\ar [r]\ar[d]_{f}&M''\ar [r]\ar[d]_{f''}&0\\0\ar[r]&N'\ar [r]&N\ar [r]&N''\ar [r]&0}$$
be a commuting diagram in $\Ch_{R}$, where the rows are degreewise-split, exact sequences. If $f'$ and $f''$ are chain homotopy equivalences, then so is $f$.
\end{lemma}

\begin{proof} Recall that a chain map is a chain homotopy equivalence if and only if its mapping cone is contractible, i.e., chain homotopy equivalent to 0. The hypotheses above imply that there is a degreewise-split short exact sequence of mapping cones
$$\xymatrix{0\ar[r]&C(f')\ar [r]&C(f)\ar [r]&C(f'')\ar [r]&0}$$
in which $C(f')$ and $C(f'')$ are contractible. It follows by \cite[Lemma 3.10]{barthel-may-riehl} that $C(f)$ is contractible and thus that $f$ is a chain homotopy equivalence.
\end{proof}

\begin{notation}\label{notn:ch-cx} Throughout this section, we simplify notation somewhat and let $-\otimes-$ denote tensoring over $R$.

Let $A$ be an algebra in $\Ch_{R}$. For any $A$-module $M$ with differential $d$, we let $\si M$ denote the $A$-module with $(\si M)_{n}\cong M_{n+1}$ for all $n$, where the element of $(\si M)_{n}$ corresponding to $x\in M_{n+1}$ is denoted $\si x$.  The differential on $\si M$ is given in terms of that on $M$ by $d(\si x)=-\si (dx)$.  The \emph{contractible path-object} on $M$ is the $A$-module
$$\operatorname{Path}(M)=(M\oplus \si M, D)$$
where $Dx=dx-\si x$ and $D\si x=-\si(dx)$, which comes equipped with a natural degreewise-split surjection of $A$-modules $\operatorname{Path}(M)\to M$.  Note that $\operatorname{Path}(M)$ is indeed contractible.
\end{notation}

\begin{remark}  There is also a Hurewicz-type model structure on the category of non-negatively graded chain complexes over an arbitrary commutative ring, as can be seen easily by inspection of the proofs in \cite{barthel-may-riehl}.  All of the results in this section have analogues in this bounded-below context, providing an unstable framework to which the methods and results apply.
\end{remark}

\subsection {Homotopical Morita theory of unbounded chain complexes}
Let $A$ be an algebra in $\Ch_{R}$. As shown in \cite[Theorems 4.5, 4.6, and 6.12]{barthel-may-riehl},  the category $(\Ch_{R})_{A}$ admits a proper, monoidal model category structure right-induced from the Hurewicz structure on $\Ch_{R}$ by the adjunction
$$\bigadjunction{\Ch_{R}}{(\Ch_{R})_{A},}{-\tensor A}{\UU}$$ 
which we call the \emph{relative model structure}.  A morphism of $A$-modules is a weak equivalence (respectively, fibration) in the relative structure if and only if the underlying morphism of chain complexes is a chain homotopy equivalence  (respectively, a degreewise-split surjection). We call the distinguished classes with respect to the relative model structure \emph{relative weak equivalences}, \emph{relative fibrations}, and \emph{relative cofibrations}, and $A$-modules that are cofibrant with respect to the relative model structure are called \emph{relative cofibrant}. The category of left modules admits an analogous relative structure.

Barthel, May, and Riehl provided the following useful characterization of relative cofibrant objects in  ${}_{A}(\Ch_{R})$.  A similar result holds for $(\Ch_{R})_{A}$.

\begin{proposition}\cite[Theorem 9.20]{barthel-may-riehl}\label{prop:bmr-cofib} An object $M$ in ${}_{A}(\Ch_{R})$ is relative cofibrant  if and only if it is a retract of an $A$-module $N$ that admits a filtration
$$0=F_{-1}N\subseteq F_{0}N\subseteq \cdots \subseteq F_{n}N\subseteq F_{n+1}N \subseteq \cdots$$
where $N=\bigcup_{n\geq 0}F_{n}N$ and for each $n\geq 0$, there is chain complex $X(n)$ with 0 differential such that $F_{n}N/F_{n-1}N\cong A \otimes X(n)$.
\end{proposition}

Barthel, May, and Riehl call  filtrations of this sort \emph{cellularly $r$-split} and show that the  inclusion maps $F_{n-1}N\to F_{n}N$ are split as nondifferential, graded $A$-modules (cf.~\cite[Definition 9.17]{barthel-may-riehl}). Note in particular that $A$ itself is always cofibrant as a right or left $A$-module.

The fact that the contractible path-object construction on a relative cofibrant module is also relative cofibrant turns out to be important in the next section, and its proof provides a nice example of the utility of Proposition \ref{prop:bmr-cofib}.

\begin{lemma}\label{lem:path} If a right $A$-module $N$ is relative cofibrant, then so is the associated contractible path-object $\operatorname{Path}(N)$.
\end{lemma}

\begin{proof}  Because the contractible path-object construction is natural, it is enough to establish this result for $A$-modules admitting a cellularly $r$-split filtration. If $0=F_{-1}N\subseteq F_{0}N\subseteq \cdots \subseteq F_{n}N\subseteq F_{n+1}N \subseteq \cdots$ is a cellularly $r$-split filtration of $N$, then it is obvious that 
$$0=\si F_{-1}N\subseteq \si F_{0}N\subseteq \cdots \subseteq \si F_{n}N\subseteq \si F_{n+1}N \subseteq \cdots$$ 
is a cellularly $r$-split filtration of $\si N$.  We can then define a cellularly $r$-split filtration of $\operatorname{Path}(N)$ by
$$F_{2n}\operatorname{Path}(N)= \si F_{n} N$$
and
$$F_{2n+1}\operatorname{Path}(N)=\operatorname{Path}(F_{n}N)$$
for all $n\geq 0$.
\end{proof}

The following special class of cofibrant modules appears naturally in our analysis of the conditions in Definition \ref{defn:special-modules}.

\begin{definition}  An object $N$ in ${}_{A}(\Ch_{R})$ is \emph{flat-cofibrant} with respect to the relative model structure if it is a retract of an $A$-module $N$ that admits a cellularly $r$-split filtration
$$0=F_{-1}N\subseteq F_{0}N\subseteq \cdots \subseteq F_{n}N\subseteq F_{n+1}N \subseteq \cdots$$
with  $F_{n}N/F_{n-1}N\cong A \otimes X(n)$ where $X(n)$ is degreewise $R$-flat, which we call a \emph{cellularly $r$-split flat filtration}.

\end{definition}

Note for any morphism of algebras $\vp: A \to B$ and for any $n\geq 1$, the bimodule $\bigoplus_{i=1}^{n}{}_{A}B_{B}$ is strictly dualizable, with dual $\bigoplus_{i=1}^{n}{}_{B}B_{A}$.

Before formulating homotopical Morita theory in $\Ch_{R}$, we provide examples of classes of modules satisfying the conditions of  Definitions \ref{defn:special-modules} and \ref{defn:a-flat}.

\begin{proposition}\label{prop:ch-special-modules} Let $A$ be an algebra in $\Ch_{R}$, and let $N$ be a left $A$-module.
\begin{enumerate}
\item If $N$ is cofibrant in the relative structure on ${}_{A}(\Ch_{R})$, then it is  homotopy flat and homotopy projective.    In particular, the category ${}_{A}(\Ch_{R})$ satisfies the CHF hypothesis (Definition \ref{defn:CHF}). 
\item If $N$ is flat-cofibrant in the relative structure on ${}_{A}(\Ch_{R})$, then it is flat and therefore strongly homotopy flat.
\item If $N$ contains $A$ as a summand, then $N$ is homotopy faithful and homotopy cofaithful.  
\item Every algebra morphism $A\xrightarrow\simeq B$ with underlying chain homotopy equivalence is a pure weak equivalence of $A$-modules.
\end{enumerate}
\end{proposition}

\begin{proof}
(1) Since a retract of a homotopy flat object is necessarily homotopy flat, it suffices to establish this result for modules admitting a cellularly $r$-split filtration. We show first that $A$-modules of the form $A\otimes X$ are homotopy flat with respect to the relative model structure, for any chain complex $X$. 

Let $f:M\to N$ be a relative weak equivalence in  $(\Ch_{R})_{A}$. The map induced by $f$,
$$f\otimes _{A}1:M\otimes _{A}(A\otimes X) \to N\otimes _{A}(A\otimes X),$$
is isomorphic to
$$f\otimes 1: M\otimes X\to N\otimes X.$$
Since the chain map underlying $f$ is a chain homotopy equivalence, and tensoring with identity morphisms preserves chain homotopy equivalences, it follows that $f\otimes_{A} 1$ is a chain homotopy equivalence, i.e., $f\otimes_{A} 1$ is a relative weak equivalence.

Suppose now that $N$ admits a cellularly $r$-split filtration $0=F_{-1}N\subseteq F_{0}N\subseteq \cdots \subseteq F_{n}N\subseteq F_{n+1}N \subseteq \cdots$.  We show by induction that each $F_{n}N$ is homotopy flat.   For $n=-1$, the claim is obviously true. For the induction step, suppose that $F_{n-1}N$ is homotopy flat, and consider the short exact sequence of right A-modules
\begin{equation}\label{eqn:ses}\xymatrix{0\ar [r]& F_{n-1}N\ar [r]& F_{n}N\ar[r]&F_{n}N/F_{n-1}N\ar[r]&0,}\end{equation}
which is split as a sequence of nondifferential, graded left A-modules. After tensoring over $A$ with any left $A$-module, it becomes a degreewise-split exact sequence of chain complexes. The rightmost term in sequence (\ref{eqn:ses}) is homotopy flat by the argument above, and the leftmost term is homotopy flat by our inductive hypothesis.  The Homotopy Five Lemma (Lemma \ref{lem:h5}) therefore implies that if we tensor sequence (\ref{eqn:ses}) with any $f:M\to M'$ in $(\Ch_{R})_{A}$ that is a relative weak equivalence, then $f\otimes_{A}1: M\otimes_{A}F_{n} N\to M'\otimes_{A}F_{n}N$ is a chain homotopy equivalence, i.e., $F_{n}N$ is homotopy flat, as desired.

To show that $N$ itself is homotopy flat,  consider a relative weak equivalence $f:M\to M'$ of right $A$-modules. Since  the maps $M\otimes_{A} F_{n-1}N\to M\otimes_{A}F_{n}N$ are degreewise-split monomorphisms of R-chain complexes, i.e., cofibrations in the Hurewicz model structure on $\Ch_{R}$, it follows that $M\otimes_{A}N=\operatorname{hocolim}_{n} M\otimes_{A}F_{n}N$, and similarly for $M'\otimes_{A}N$. Moreover, the map $f\otimes_{A}1:M\otimes_{A}N \to  M'\otimes_{A}N$ is the map induced on homotopy colimits by the sequence of maps  $$f\otimes_{A}1:M\otimes_{A}F_{n}N \to  M'\otimes_{A}F_{n}N,$$ all of which are weak equivalences in the Hurewicz model structure on $\Ch_{R}$.  It follows that $f\otimes_{A}1:M\otimes_{A}N \to  M'\otimes_{A}N$ is also a weak equivalence in the Hurewicz model structure.

The proof that cofibrant modules are homotopy projective in the relative model structure on ${}_{A}(\Ch_{R})$ proceeds by a highly analogous inductive argument, which we leave to the reader.

\noindent (2) Given (1), to show that any flat-cofibrant module $N$ is strongly homotopy flat, it suffices to check that the functor $-\otimes_{A}N: (\Ch_{R})_{A}\to \Ch_{R}$ preserves finite products and kernels. It is immediate that $-\otimes_{A}N$ preserves finite products for every left $A$-module $N$, since they are the same as finite sums in $\Ch_{R}$.   Preservation of kernels is equivalent to flatness (Definition \ref{defn:a-flat}) in $(\Ch_{R})_{A}$, since it is an abelian category.    

Since any retract of a strongly homotopy flat module is strongly homotopy flat, it is enough to show that if $N$ admits a cellularly $r$-split flat filtration, 
$$0=F_{-1}N\subseteq F_{0}N\subseteq \cdots \subseteq F_{n}N\subseteq F_{n+1}N \subseteq \cdots$$
with  $F_{n}N/F_{n-1}N\cong A \otimes X(n)$, then $N$ is strongly homotopy flat.  Observe first that if $X$ is degreewise $R$-flat, then the functor $-\otimes_{A}(A\otimes X)$ preserves kernels, since it is isomorphic to $-\otimes X$.    We prove by induction on $n$ that $F_{n}N$ is strongly homotopy flat for all $n$.  

The base of the induction is trivial, since $F_{-1}N=0$.  Suppose that $F_{n-1}N$ is strongly homotopy flat, for some $n\geq 0$, and consider the degreewise $A$-split sequence
\begin{equation}\label{eqn:ses2}\xymatrix{0\ar [r]& F_{n-1}N\ar [r]& F_{n}N\ar[r]&F_{n}N/F_{n-1}N\ar[r]&0.}\end{equation}
Tensoring sequence (\ref{eqn:ses2}) with any short exact sequence 
$$\xymatrix{0\ar [r]& K\ar [r]& M\ar[r]&M'\ar[r]&0}$$
of right $A$-modules gives rise to a commuting diagram of chain maps
$$\xymatrix{&0\ar [d]&&0\ar[d]\\
0\ar [r]& K\otimes_{A}F_{n-1}N\ar [r]\ar [d]& K\otimes_{A}F_{n}N\ar[r]\ar [d]&K\otimes _{A}F_{n}N/F_{n-1}N\ar[r]\ar [d]&0\\
0\ar [r]& M\otimes_{A}F_{n-1}N\ar [r]\ar [d]& M\otimes_{A}F_{n}N\ar[r]\ar [d]&M\otimes _{A}F_{n}N/F_{n-1}N\ar[r]\ar [d]&0\\
0\ar [r]& M'\otimes_{A}F_{n-1}N\ar [r]\ar [d]& M'\otimes_{A}F_{n}N\ar[r]\ar [d]&M'\otimes _{A}F_{n}N/F_{n-1}N\ar[r]\ar [d]&0\\
&0&0&0}$$
in which all rows are split exact, and the leftmost and rightmost columns are exact.  The Five Lemma then implies that the map $K\otimes_{A}F_{n}N\to M\otimes_{A}F_{n}N$ is injective as desired, i.e., $-\otimes_{A}F_{n}N$ commutes with kernels. 

To see that $-\otimes _{A}N$ commutes with kernels, observe that for any morphism $f:M\to M'$ of right $A$-modules
\begin{align*}(\ker f) \otimes _{A }(\operatorname{colim}_{n}F_{n}N)&\cong \operatorname{colim}_{n} (\ker f\otimes_{A}F_{n}N)\cong \operatorname{colim}_{n} \ker (f\otimes_{A}1_{F_{n}N})\\
& \cong \ker \operatorname{colim}_{n}(f\otimes_{A}1_{F_{n}N})\cong\ker (f\otimes_{A}1_{N}),
\end{align*}
where the first isomorphism follows from the fact that $\ker f\otimes_{A}- $ is a left adjoint, the second isomorphism from the inductive argument, and the third isomorphism from the fact that $N$ is the colimit of a directed system of monomorphisms.

\noindent (3) This is just a special case of Remark \ref{rmk:retract-faithful}.

\noindent (4) Since the CHF hypothesis is satisfied, this follows from Proposition \ref{prop:pure we alt}.
\end{proof}

\begin{remark}
The hypothesis in $(3)$ can certainly be weakened, but we will not pursue a complete classification of homotopy (co)faithful $A$-modules here. For instance, one can show that every faithfully flat $R$-module $M$ is homotopy faithfully flat when viewed as an object of $\Ch_R$ in the Hurewicz model structure. Since homotopy faithfulness is preserved by chain homotopy equivalences, and since homotopy faithfulness is inherited from retracts, it follows that every split $R$-chain complex $V$ that has at least one faithfully flat homology module is homotopy faithful as an object of $\Ch_R$. Note also that $A\tensor V$ is homotopy faithful as an $A$-module whenever $V$ is a homotopy faithful object of $\Ch_R$.

Let $\vp:A\to B$ be a morphism of algebras in $\Ch_{R}$.   Note that for $A$ to be a retract of $B$ as $A$-modules, it must be true that $\vp$ induces an injection in homology.
\end{remark}

The next theorem is an immediate consequence of Theorem \ref{thm:morita} and Proposition \ref{prop:ch-special-modules}.

\begin{theorem}
Let $A$ and $B$ be algebras in $ \Ch_{R}$, and let $X$ be an $A$-$B$-module such that the underlying right $B$-module admits $B$ as a retract and is a retract of a cellularly $r$-split $B$-module.  The Quillen adjunction
\begin{equation} \label{eq:quiad2}
\bigadjunction{(\Ch_{R})_A}{(\Ch_{R})_B}{-\tensor_A X}{\Map_B(X,-)}
\end{equation}
is a Quillen equivalence if and only if
\begin{enumerate}
\item the map $\eta_A\colon A\rightarrow \Map_B(X,X)$ is a relative weak equivalence; and
\item the bimodule $X$ is homotopy right dualizable.
\end{enumerate}
\end{theorem}

\begin{example}  The theorem above implies a homotopical version of the usual Morita equivalence between a ring $R$ and the ring of $(n\times n)$-matrices with coefficients in $R$. Let  $B$ be any algebra in $\Ch_{R}$, and  let $X=B^{\oplus n}$ for some $n\in \mathbb N$.  For any weak equivalence of algebras $\vp\colon A \xrightarrow \simeq \Map_{B}(X,X)$, the adjunction governed by $X$
$$\bigadjunction{(\Ch_{R})_A}{(\Ch_{R})_B}{-\tensor_A X}{\Map_B(X,-)},$$
is a Quillen equivalence, where the $A$-module structure on $X$ is encoded by $\vp$.
\end{example}

\subsection{Homotopical Morita theory of  differential graded comodules}
The existence of a model category structure for categories of comodules over corings in the context of unbounded chain complexes was established in \cite{hkrs}.

\begin{theorem}\cite[Theorem 6.6.3]{hkrs} \label{thm:hkrs} Let $R$ be any commutative ring. For any algebra $A$ in $\Ch_{R}$ and any $A$-coring $C$, the category $(\Ch_{R})_A^C$ of $C$-comodules in $A$-modules admits a model category structure left-induced from the relative model structure on $(\Ch_{R})_{A}$ via the forgetful functor.   
\end{theorem}

We call the distinguished classes of morphisms in this left-induced model structure \emph{relative weak equivalences}, \emph{relative fibrations}, and \emph{relative cofibrations} of $C$-comodules.   

This model category structure on $(\Ch_{R})_A^C$ admits particularly nice fibrant replacements, which prove useful in our analysis of homotopical Morita theory below. We recall first the well known definition of the cobar construction.  

\begin{notation} Let $T_{A}$ denote the \emph{free tensor algebra over $A$} functor, which to any $A$-bimodule $X$ associates the graded $A$-algebra $T_{A}X=A \oplus \bigoplus _{n\geq 1} X^{\otimes_{A} n}$, the homogeneous elements of which are denoted $x_{1}|\cdots |x_{n}$.

Let  $(C,\Delta, \ve, \eta)$ be a coaugmented $A$-coring, i.e., $(C,\Delta, \ve)$ is an $A$-coring equipped with a morphism $\eta: A \to C$ of $A$-corings, with coaugmentation coideal $\overline C=\operatorname{coker} (\eta\colon   A \to C)$.  We use the Einstein summation convention and  write $\Delta (c)= c_{i}\otimes c^{i}$ for all $c\in C$ and similarly for the map induced by $\Delta$ on  $\overline C$.  If $(M, \rho)$ is a right $C$-comodule, then we apply the same convention again and write $\rho(x)=x_{i}\otimes c^{i}$ for all $x\in M$, and similarly for a left $C$-comodule.
\end{notation}

\begin{definition} Let  $(C,\Delta, \ve, \eta)$ be a coaugmented $A$-coring. For any right $C$-comodule $(M,\rho)$ and left $C$-comodule $(N, \lambda)$, we let $\Om_{A}(M;C;N)$ denote the $A$-bimodule
$$(M\otimes_{A} T_{A} (s^{-1}  \overline C) \otimes_{A} N, d_{\Om}),$$
where
\begin{align*}
d_{\Om}(x \otimes \si c_{1}|\cdots |\si c_{n}\otimes y)= &\,dx \otimes \si c_{1}|\cdots |\si c_{n}\otimes y\\
&+x\otimes \sum _{j=1}^{n}\pm\si c_{1}|\cdots |\si dc_{j}|\cdots \si c_{n}\otimes y\\
&\pm x \otimes \si c_{1}|\cdots |\si c_{n}\otimes dy\\
&\pm x_{i}\otimes \si c^{i}| \si c_{1}|\cdots |\si c_{n}\otimes y\\
&+x\otimes \sum _{j=1}^{n}\pm\si c_{1}|\cdots |\si c_{j,i}|\si c_{j}^{{i}}|\cdots \si c_{n}\otimes y\\
&\pm x\otimes \si c_{1}|\cdots |\si c_{n}|\si c_{i}\otimes y^{i}
\end{align*}
where all signs are determined by the Koszul rule, the differentials of $M$, $N$, and $C$ are all denoted $d$, and $\si 1=0$ by convention. 

If  $N=C$, then $\Om_{A} (M;C;C)$ admits a right $C$-comodule structure induced from the rightmost copy of $C$.  
\end{definition}

\begin{remark}\label{rmk:coinv} The cobar construction $\Om_{A} (M;C;C)$ is a ``cofree resolution'' of $M$, in the sense that the coaction map $\rho\colon M\to M\otimes_{A} C$ factors in $(\Ch_{R})_{A}^{C}$ as
\begin{equation}\label{eqn:factor-coaction}\xymatrix{ M\ar [dr]_{\widetilde \rho}\ar [rr]^{\rho}&&M\otimes_{A} C,\\ 
&\Om_{A}(M;C;C)\ar [ur]_{q}}
\end{equation}
where $\widetilde \rho (x)=x_{i}\otimes 1\otimes c^{{i}}$  and $q(x\otimes 1\otimes c)=x\otimes c$, while $q(x \otimes \si c_{1}|\cdots |\si c_{n}\otimes c)=0$ for all $n\geq 1$.   It is well known that the composite
$$\Om_{A}(M;C;C) \xrightarrow q M\otimes_{A} C \xrightarrow {1_{M}\otimes_{A} \ve} M$$
is a chain homotopy equivalence, by a standard ``extra codegeneracy'' argument, with chain homotopy inverse $\tilde \rho\colon  M\to \Om_{A} (M;C;C)$.  
\end{remark}

The proof of Theorem 7.8 in \cite{hess-shipley} carries over in this somewhat more general setting, to enable us to construct particularly nice fibrant replacements for comodules over flat corings.



\begin{theorem}\label{thm:fib-repl-comod} Let $D$ be a flat, coaugmented $A$-coring. For any right $D$-comodule $M$, the $D$-comodule $\Om_{A} (M;D;D)$ is relative fibrant.
\end{theorem}

\begin{proof}[Proof sketch]  As explained in detail in the proof of Theorem 7.8 in \cite{hess-shipley}, the  $D$-comodule $\Om_{A} (M;D;D)$ is the inverse limit of a tower, natural in $M$, 
$$ ...\xrightarrow {q^{n+1}} E_{n}M\xrightarrow {q^{n}} E_{n-1}M\xrightarrow {q_{n-1}} ...\xrightarrow{q_{1}} E_{0}M=M\otimes_{A} D$$
 in $(\Ch_{R})_A^D$, where each morphism $q_{n}:E_{n}M\to E_{n-1}M$ is given by a pullback  in $(\Ch_{R})_A^D$ of the form
 $$\xymatrix{E_{n}M\ar [d]_{q_{n}}\ar [r]& \operatorname{Path}(B_{n}M)\otimes _{A}D\ar[d]_{p_{n}\otimes_{A}D}\\
 E_{n-1}M\ar [r]^{k_{n}}& B_{n}M\otimes _{A}D,}$$
 where $B_{n}$ is an endofunctor on the category of $A$-bimodules, and $p_{n}:\operatorname{Path}(B_{n}M) \to B_{n}M$ is the natural map from the contractible path-object on $B_{n}M$ to  $B_{n}M$ itself, which is a relative fibration of right $A$-modules.  It is important here that $D$ be flat over $A$, so that pullbacks in $(\Ch_{R})_A^D$ are created in $(\Ch_{R})_A$.  Since one shows by hand that $\Om_{A} (M;D;D)$ satisfies the universal condition to be the inverse limit of the tower of $E_{n}M$'s, it is not necessary to suppose that inverse limits are created in $(\Ch_{R})_A$.
 
It follows that $p_{n}\otimes _{A}D$ and thus $q_{n}$ are relative fibrations in $(\Ch_{R})_A^D$.   We conclude that $q:\Om_{A} (M;D;D) \to M\otimes_{A}D$ is a relative fibration, which implies that $\Om_{A}(M;D;D)$ is a relative fibrant $D$-comodule, since $M\otimes_{A}D$ is relative fibrant in $(\Ch_{R})_A^D$, as every right $A$-module is relative fibrant in $(\Ch_{R})_{A}$.
\end{proof}

We can now establish the existence of an interesting class of copure weak equivalences of corings.

\begin{theorem} \label{thm:copure} Let $A$ be an algebra in $\Ch_{R}$.  If $C$ is a flat $A$-coring, and $D$ is a coaugmented flat-cofibrant $A$-coring, then every relative weak equivalence $f\colon C \to D$ of $A$-corings  is copure.
\end{theorem}

\begin{proof} Observe first that since the model category structures on $(\Ch_{R})_{A}^{C}$ and $(\Ch_{R})_{A}^{D}$ are both left-induced from $(\Ch_{R})_{A}$, the change-of-corings adjunction 
$$
\bigadjunction{(\Ch_{R})_{A}^{C}}{(\Ch_{R})_{A}^{D}}{f_*}{-\cotensor_D C}.
$$
is a Quillen adjunction.  In particular, the functor $-\cotensor _{D}C$ preserves weak equivalences between fibrant objects, and $f_{*}$ preserves all weak equivalences.  It follows that in order to prove that the counit of the change-of-corings adjunction
$$\ve_{M}: f_{*}(M\cotensor_{D}C) \to M$$
is a relative weak equivalence for all relative fibrant $D$-comodules $M$, it is enough to consider the case where $M$ is cofibrant and fibrant.  By Theorem \ref{thm:fib-repl-comod}, it therefore suffices to show that the counit
$$f_{*}\big(\Om_A(M;D;D)\cotensor _{D}C\big) \to \Om_A(M;D;D)$$
is a relative weak equivalence for all relative fibrant and cofibrant $D$-comodules $M$.  
Since $\Om_{A} (M;D;D)$ is $D$-cofree (after forgetting differentials), one can check easily that 
$$\Om_{A} (M;D;D)\cotensor _{D}C\cong \Om_{A} (M;D;C)$$
and that under this identification, the counit map corresponds to
$$\Om_{A} (1_{M};1_{D};f): \Om_{A} (M;D;C) \to \Om_{A} (M;D;D).$$

We now prove by induction that this map is a relative weak equivalence, for $M$ relative fibrant and cofibrant, using the description of the cobar construction as a limit from the proof of Theorem \ref{thm:fib-repl-comod}.   Before starting the argument, we recall from the proof of Theorem 7.8 in \cite{hess-shipley} that the functor $B_{n}$ used in the construction of $\Om_{A}(M;D;D)$ is specified by
$$B_{n}M= s^{-n}(M\otimes_{A}\overline D^{\otimes_{A}n+1}),$$
where $s^{-n}$ denotes the $n$-fold iteration of the functor $\si$ described in Notation \ref{notn:ch-cx}.  Since $D$ is relative cofibrant as a right $A$-module, and the coaugmentation $\eta: A\to D$ is split by the counit of $D$ and is therefore a relative cofibration, the coaugmentation coideal $\overline D$ is also a relative cofibrant right $A$-module, as is its desuspension $\si \overline D$. Repeated application of Proposition \ref{prop:x} implies then that $B_{n}M$ is also relative cofibrant as a right $A$-module, since $M$ is relative cofibrant.

To initiate the induction, recall that the CHF condition holds in $(\Ch_{R})_{A}$, whence the functor $M\otimes_{A}-$ preserves relative weak equivalences, since $M$ is relative cofibrant.  In particular, $1_{M}\otimes _{A}f: M\otimes _{A}C\to M\otimes_{A}D$ is a relative weak equivalence of right $A$-modules.

For the inductive step of the argument, suppose that the counit
$$f_{*}(E_{n-1}M\cotensor _{D}C) \to E_{n-1}M$$ 
is a relative weak equivalence, for some $n\geq 1$.   Since $-\cotensor_{D}C$ is a right adjoint, and $(X\otimes _{A}D)\cotensor _{D}C\cong X\otimes_{A}C$ for every right $A$-module $X$, the $C$-comodule $E_{n}M\cotensor _{D}C$ is isomorphic to the pullback in $(\Ch_{R})_{A}^{C}$ of
$$E_{n-1}M\cotensor_{D}C\xrightarrow{k_{n}\cotensor _{D}C} B_{n}M\otimes_{A}C \xleftarrow {q_{n}\otimes _{A}C} \operatorname{Path}(B_{n}M) \otimes_{A}C.$$  Since pullbacks in $(\Ch_{R})_{A}^{C}$ and  $(\Ch_{R})_{A}^{D}$ are created in $(\Ch_{R})_{A}$, $f_{*}(E_{n}M\cotensor _{D}C)$ is isomorphic to the pullback in $(\Ch_{R})_{A}^{D}$ of
$$f_{*}(E_{n-1}M\cotensor_{D}C)\xrightarrow{k_{n}\cotensor _{D}C} f_{*}(B_{n}M\otimes_{A}C) \xleftarrow {q_{n}\otimes _{A}C} f_{*}\big(\operatorname{Path}(B_{n}M) \otimes_{A}C\big).$$

Consider the commuting diagram
$$\xymatrix{f_{*}(E_{n-1}M\cotensor_{D}C)\ar[rr]^{k_{n}\cotensor _{D}C}\ar [d]_{\simeq}&& f_{*}(B_{n}M\otimes_{A}C)\ar [d]_{\simeq}&& f_{*}\big(\operatorname{Path}(B_{n}M) \otimes_{A}C\big) \ar[ll]_{q_{n}\otimes _{A}C}\ar[d]_{\simeq}\\
E_{n-1}M\ar[rr]^{k_{n}}&& B_{n}M\otimes_{A}D&& \operatorname{Path}(B_{n}M) \otimes_{A}D \ar[ll]_{q_{n}\otimes _{A}D}}$$
in which the three vertical arrows are various components of the counit.  Note that the morphism induced on the pullbacks  is exactly the counit
$$f_{*}(E_{n}M\cotensor _{D}C) \to E_{n}M$$
because pullbacks in $(\Ch_{R})_{A}^{D}$ are created in $(\Ch_{R})_{A}$. The leftmost vertical arrow is a relative weak equivalence by the inductive hypothesis, while the two others are relative weak equivalences because CHF holds for $A$-modules.  Since $(\Ch_{R})_{A}$ is proper, and $q_{n}\otimes_{A}C$ and $q_{n}\otimes_{A}D$ are relative fibrations of $A$-modules, it follows that
$f_{*}(E_{n}M\cotensor _{D}C) \to E_{n}M$
is a relative weak equivalence as well.

We therefore have a commuting diagram of towers in $(\Ch_{R})_{A}^{D}$
$$\xymatrix{\vdots \ar[d]&\vdots\ar[d]\\
f_{*}(E_{n}M\cotensor _{D}C)\ar [r]^(0.6){\simeq}\ar[d]& E_{n}M\ar [d]\\
f_{*}(E_{n-1}M\cotensor _{D}C)\ar [r]^(0.6){\simeq}\ar[d]& E_{n-1}M\ar [d]\\
\vdots&\vdots}$$
such that the maps in the underlying towers in $(\Ch_{R})_{A}$ are relative fibrations and thus their limits are in fact homotopy limits.  Since countable inverse limits in $(\Ch_{R})_{A}^{D}$ are created in $(\Ch_{R})_{A}$, the induced map 
$$\lim_{n} f_{*}(E_{n}M\cotensor_{D}C) \to \lim _{n}E_{n}M=\Om_A(M;D;D)$$ 
is a relative weak equivalence as well.  To conclude observe that by modifying slightly the last two paragraphs of the proof of Theorem 7.8 in \cite{hess-shipley}, one can prove by hand that $f_{*}\Om_A(M;D;C)$ satisfies the universal condition to be $\lim_{n} f_{*}(E_{n}M\cotensor_{D}C)$.
\end{proof}

The next theorem provides an illustration of homotopy faithfully flat descent in the context of unbounded chain complexes.

\begin{theorem} \label{thm:homotopic descent for chains} Let $A$ and $B$ be algebras in $\Ch_{R}$, $C$  a flat $A$-coring, and 
${}_AX_B$  a right dualizable bimodule. If as a left $A$-module $X$ is flat-cofibrant and contains $A$ as a retract, then $X$ satisfies effective homotopic descent with respect to $(A,C)$.
\end{theorem}

\begin{proof}   Parts (2) and (3) of Proposition \ref{prop:ch-special-modules} together imply that this result follows from Theorem \ref{thm:homotopic descent}.
\end{proof}

We now formulate homotopical Morita theory for differential graded corings in the special case of adjunctions induced by a morphism of corings.  The general case is not harder to formulate, but perhaps less transparent.

\begin{theorem} \label{thm:we=QE for dg-corings}
Let $(\varphi,f)\colon (A,C)\rightarrow (B,D)$ be a morphism of flat corings in $\Ch_{R}$ such that, as a left $A$-module, $B$ is flat-cofibrant and contains $A$ as a retract and such that $D$ is coaugmented and flat-cofibrant as a left $B$-module.

The adjunction governed by $(\varphi,f)$,
$$\bigadjunction{(\Ch_{R})_A^C}{(\Ch_{R})_B^D}{(T_{\vp, f})_{*}}{(T_{\vp, f})^{*}},$$
is a Quillen equivalence if and only if the morphism of $B$-corings $f\colon B_*(C) \rightarrow D$ is a relative weak equivalence.
\end{theorem}

\begin{proof}  To see that Corollary  \ref{thm:we=QE for corings} implies the desired equivalence, observe that 
\begin{itemize}
\item all objects in $\Ch_{R}$ are fibrant, and every algebra is cofibrant as a module over itself;
\item $B$ is cofibrant and strongly homotopy flat as a left $A$-module by Proposition \ref{prop:ch-special-modules} (2);
\item by Theorem \ref{thm:homotopic descent for chains}, $B$ satisfies effective descent with respect to $C$; and
\item by Theorem \ref{thm:copure}, the morphism of $B$-corings $f:B_{*}(C)\to D$ is copure, since the fact that  $C$ is left $A$-flat implies that $B_*(C)$ is left $B$-flat.
\end{itemize}
\end{proof}

\begin{remark}
In particular, if we let $f$ be the identity map of $B_*(C)$, then the adjunction
$$\bigadjunction{(\Ch_{R})_{A}^{C}}{(\Ch_{R})_{B}^{B_*(C)}}{(T_{\vp, 1})_{*}}{(T_{\vp, 1})^{*}}$$
is a Quillen equivalence if as a left $A$-module, $B$ is flat-cofibrant and contains $A$ as a retract, and $C$ is coaugmented and flat-cofibrant as a left $A$-module.
\end{remark}

\begin{remark}  The hypothesis on $\vp\colon A \to B$ in the Theorem \ref{thm:we=QE for dg-corings}, requiring that $B$ be flat-cofibrant and contain $A$ as a retract, at least as a left $A$-module, is not too restrictive.  For example, the \emph{KS-extensions} (also known as \emph{relative Sullivan algebras}) of rational homotopy theory \cite{fht} are classical examples of algebra morphisms such that the target is flat-cofibrant over the source. More generally, any algebra morphism in $\Ch_{R}$ can be replaced up to quasi-isomorphism by an algebra morphism satisfying this condition.  Indeed, every morphism $\vp:A \to B$ admits a factorization
$$\xymatrix{(A,d) \ar[rr]^{\vp} \ar [dr]_{j}&&(B,d),\\&(A\coprod TV, D)\ar [ur]_{p}^{\sim}}$$
where $j$ is the inclusion into a free extension of $(A,d)$ by a free algebra on a degreewise $R$-free graded module $V$, and $p$ is a quasi-isomorphism, and \cite[Proposition 4.3.11, Remark 4.3.12]{karpova} implies that $(A\coprod TV, D)$  admits a cellularly $r$-split flat filtration and is therefore flat-cofibrant as a left $A$-module.  The inclusion $j$ can be chosen to admit a retraction as long as $\vp$ induces an injection in homology.
\end{remark}

\begin{example} {{{}}}
Recall the descent coring associated to a morphism of algebras from Example \ref{ex:desc-coring}. Let 
$$\xymatrix{A\ar[r]^{\vp}\ar[d]_{\alpha}&B\ar[d]_{\beta}\\ A'\ar[r]^{\vp'}&B'}$$
be a commuting diagram of algebra morphisms in $\Ch_{R}$, which induces a morphism of descent corings 
$$(\alpha, \beta\otimes_{\alpha}\beta): (A, B\otimes _{A}B) \to (A', B'\otimes_{A'}B').$$
If, as  a left $A$-module,  $A'$ is flat-cofibrant and contains $A$ as a retract, $B$ is flat over $A$, and $B'$ is flat-cofibrant over $A'$, then the associated adjunction 
$$\bigadjunction{(\Ch_{R})_A^{B\otimes_{A}B}}{(\Ch_{R})_{A'}^{B'\otimes _{A'}B'}}{(T_{\alpha, \beta\otimes_{\alpha}\beta})_{*}}{(T_{\alpha, \beta\otimes_{\alpha}\beta})^{*}}$$
is a Quillen equivalence if and only if $\beta\otimes_{\alpha}\beta:B\otimes_{A}B\to B'\otimes_{A'}B'$ is a relative weak equivalence.
\end{example}

In \cite{berglund-hess:hhg}  we provide further concrete applications of Theorem \ref{thm:we=QE for dg-corings} related to the theory of homotopic Hopf-Galois extensions.

\appendix

\section{Enriched model categories}\label{app:enriched}
In this appendix we review some elementary aspects of enriched model category theory. For a thorough treatment, we refer to Riehl \cite{riehl}.

Let $(\VV,\tensor,\kk)$ be a closed symmetric monoidal category. A category $\CC$ has a \emph{$\VV$-structure} if it is tensored, cotensored and enriched in $\VV$. For objects $X,Y\in \CC$ and $K\in \VV$, we use the notation
$$K\tensor X\in\CC,\quad \Map_\CC(X,Y)\in\VV,\quad Y^K\in\CC,$$
for the tensor product, enrichment and cotensor product, respectively. We assume that the structures are compatible in the sense that there are natural bijections
$$\CC(K\tensor X,Y) \cong \VV(K,\Map_\CC(X,Y)) \cong \CC(X,Y^K).$$

An adjunction between $\VV$-categories that preserves all relevant structure is called a \emph{$\VV$-adjunction}. {As is well known,} the following proposition characterizes $\VV$-adjunctions.

\begin{proposition} \label{prop:V-adjunction}
The following are equivalent for an adjunction
$$\adjunction{\CC}{\DD}{F}{G},\quad F \dashv G,$$
between $\VV$-categories.
\begin{enumerate}
\item The left adjoint $F$ is a tensor functor, i.e., there is a natural isomorphism $\alpha_{K,X}\colon F(K\tensor X) \stackrel{\cong}{\rightarrow} K\tensor F(X)$ for $K\in\VV$ and $X\in \CC$, such that
$$\alpha_{K\otimes L, X}=(1\otimes \alpha_{L,X})\alpha_{K, L\otimes X},$$
for all $K,L\in \VV$ and $X\in \CC$.
\item The right adjoint $G$ is a cotensor functor, i.e., there is a natural isomorphism $\beta^{K,Y}\colon (GY)^K\stackrel{\cong}{\rightarrow} G(Y^K)$ for $K\in\VV$ and $Y\in \DD$, satisfying a similar associativity relation.
\item The adjunction is $\VV$-enriched, i.e., there is a natural isomorphism $$\phi^{X,Y}\colon \Map_{\DD}(FX,Y)\stackrel{\cong}{\rightarrow} \Map_{\CC}(X,GY)$$ for $X\in\CC$ and $Y\in\DD$.
\end{enumerate}
\end{proposition}

\begin{proof}
Use the Yoneda Lemma and the diagram
$$
\xymatrix@C=-10pt@R=20pt{
\DD(K\tensor FX,Y) \ar@{-->}[d]^-{\alpha_{K,X}^*} \ar[rr]^-\cong && \VV(K,\Map_\DD(FX,Y)) \ar@{-->}[d]^-{\phi_*^{X,Y}} && \DD(FX,Y^K) \ar[d]^-\cong \ar[ll]_-\cong \\
\DD(F(K\tensor X),Y) \ar[dr]^-\cong && \VV(K,\Map_\CC(X,GY) \ar[dl]_-\cong \ar[dr]^-\cong && \CC(X,G(Y^K)) \\
& \CC(K\tensor X,GY) && \CC(X,(GY)^K) \ar@{-->}[ur]^-{\beta_*^{K,Y}}}
$$
to prove the proposition.
\end{proof}

\subsection{$\VV$-model categories}
Suppose that $(\VV,\tensor,\kk)$ is a closed symmetric monoidal category that also admits the structure of a Quillen model category. 

\begin{definition}
A $\VV$-category $\CC$ with a model structure is called a \emph{$\VV$-model category} if the following axiom is satisfied.

\begin{axiom} \label{axiom:sm7}
Given morphisms
$$i\colon A\rightarrow B\in \CC,\quad j\colon K \rightarrow L\in\VV,\quad p\colon X \rightarrow Y\in \CC,$$
such that $i$ and $j$ are cofibrations, $p$ is a fibration, and at least one out of $i$, $j$, $p$ is a weak equivalence, the following lifting problems can be solved.

\begin{enumerate}
\item (In terms of the $\VV$-enrichment of $\CC$)
$$\xymatrix{K \ar[d]^-j \ar[r] & \Map_{\CC}(B,X) \ar[d]^-{i*p} \\ L \ar[r] \ar@{-->}[ur] & \Map_{\CC}(A,X)\times_{\Map_{\CC}(A,Y)} \Map_{\CC}(B,Y);}$$

\item (In terms of the $\VV$-tensor structure on $\CC$)
$$\xymatrix{K\tensor B \bigsqcup_{K\tensor A} L\tensor A \ar[r] \ar[d]^-{j*i} & X \ar[d]^-p \\ L\tensor B \ar@{-->}[ur] \ar[r] & Y;}$$

\item (In terms of the $\VV$-cotensor structure on $\CC$)
$$\xymatrix{A\ar[d]^-i \ar[r] & X^L \ar[d]^-{j*p} \\ B \ar[r] \ar@{-->}[ur] & X^K \times_{Y^K} Y^L.}$$
\end{enumerate}
\end{axiom}

It is an exercise in adjunctions to show that the lifting problems are equivalent: a solution to one yields solutions to the two other upon taking appropriate adjoints.

The symmetric monoidal category $\VV$, with its given model structure and its canonical $\VV$-structure, is called a \emph{symmetric monoidal model category} if it is a $\VV$-model category itself, cf.~\cite{schwede-shipley}.
\end{definition}

\subsection{$\VV$-enrichment of induced model structures}
Consider an adjunction between model categories
\begin{equation} \label{eq:qad}
\adjunction{\CC}{\DD}{F}{G},
\end{equation}
{where $F$ is the left adjoint.} We say that the model structure on $\DD$ is \emph{right induced} from $\CC$ if a map $f$ in $\DD$ is a weak equivalence (fibration) if and only if $Gf$ is a weak equivalence (fibration) in $\CC$. Dually, we will say that the model structure on $\CC$ is \emph{left induced} from $\DD$ if a map $f$ in $\CC$ is a weak equivalence (cofibration) if and only if $Ff$ is a weak equivalence (cofibration) in $\DD$.

The following proposition is presumably well-known (see e.g.~\cite[Lemma 2.25]{bhkkrs} for the case of left-induced structures), but we indicate the proof for the reader's convenience.

\begin{proposition} \label{prop:v-induced}
A model structure induced from a $\VV$-model structure along a $\VV$-adjunction is itself a $\VV$-model structure.

More precisely, suppose given an adjunction between model categories as in \eqref{eq:qad}, where $\CC$ and $\DD$ have $\VV$-structures and the adjunction has a $\VV$-structure.

\begin{enumerate}
\item If the model structure on $\DD$ is right induced from $\CC$, and $\CC$ satisfies Axiom \ref{axiom:sm7}, then so does $\DD$.
\item If the model structure on $\CC$ is left induced from $\DD$, and $\DD$ satisfies Axiom \ref{axiom:sm7}, then so does $\CC$.
\end{enumerate}
\end{proposition}

\begin{proof}
Given a cofibration $j\colon K\rightarrow L$ in $\VV$ and a fibration $p\colon X\rightarrow Y$ in $\DD$, we need to show that $j*p \colon X^L\rightarrow X^K \times_{Y^K} Y^L$ is a fibration and that it is a weak equivalence if either $j$ or $p$ is a weak equivalence. The map $Gp$ is a fibration in $\CC$, and since $\CC$ satisfies Axiom \ref{axiom:sm7}, the map $j*Gp$ is a fibration. Since $G$ is a cotensor functor, there is an isomorphism of morphisms $G(j*p) \cong j*Gp$, whence $G(j*p)$ is a fibration. Since the model structure on $\DD$ is right induced from the model structure on $\CC$, this means that $j*p$ is a fibration. We leave the rest of the proof to the reader.
\end{proof}

\section{Dualizability}\label{app:dualizable}
The notion of dualizability plays an important role in our study of those (braided) bimodules that induce Quillen equivalences between model categories of modules over algebras and of comodules over corings.  We recall here this classical notion and some of its elementary properties, expressed in terms of adjunctions in bicategories. We do not recall the definition of a bicategory, which the reader can find in \cite{benabou}. The following definition generalizes the usual notion of an adjunction of categories.

\begin{definition}\label{defn:dual}  Let $\mathsf{C}$ be a bicategory. An adjunction in $\mathsf C$ consists of a pair of objects $A$ and $B$, a pair of $1$-morphisms $l:A \to B$ and $r:B\to A$, and a pair of $2$-morphisms $\eta: 1_{A}\to rl$ and $\ve: lr \to 1_{B}$ satisfying the \emph{triangle identities}
$$(r*\ve)(\eta*r)=1_{r}\quad\text{and}\quad (\ve*l)(l*\eta)=1_{l},$$
where $*$ denote the usual whiskering of $2$-morphism by a $1$-morphism.
We call $l$ the \emph{left adjoint}, $r$ the \emph{right adjoint}, $\eta$ the \emph{unit}, and $\ve$ the \emph{counit} of the adjunction, and write $l\dashv r$.
\end{definition}

\begin{remark} The right adjoint of a $1$-morphism is unique up to isomorphism if it exists.  Moreover, bifunctors clearly preserve adjunctions.
\end{remark}

\begin{definition}  Let $(\VV, \tensor, \Bbbk)$ be a monoidal category.  Its \emph{delooping bicategory} $\mathsf B\VV$  is the bicategory with exactly one object $\bullet$ and such that the category $\mathsf B\VV (\bullet,\bullet)$ is $\VV$, where composition of $1$-morphisms in $\mathsf B\VV$ is given by the tensor product of objects in $\VV$, and composition of $2$-morphisms in $\mathsf B\VV$ is the same as composition of morphisms in $\VV$.
\end{definition}

\begin{definition} Let $(\VV, \tensor, \Bbbk)$ be a monoidal category. 
An object $X$ in $\VV$ is \emph{right dualizable} if, seen as a $1$-morphism in $\mathsf B \VV$, it admits a right adjoint $Y$, which we call a \emph{right dual} of $X$, while $X$ is a \emph{left dual} to $Y$.
\end{definition}

\begin{remark}  Unraveling the definition above, we see that $Y$ is a right dual to $X$ if there are morphisms
$$u\colon \Bbbk \to X \otimes Y \quad \text{and}\quad e\colon Y\otimes X \to \Bbbk,$$
which we call the \emph{coevaluation} and \emph{evaluation} in order to distinguish them from the numerous other units and counits of adjunctions with which we work in this paper, such that the composites
$$X \xrightarrow {u\otimes 1} X\otimes Y\otimes X \xrightarrow {1\otimes e} X\quad\text{and}\quad Y \xrightarrow{1\otimes u} Y\otimes X \otimes Y \xrightarrow {e\otimes 1} Y$$
are both identities.
\end{remark}

%

The following well known example of a bicategory, in which the notion of adjunction is a many-object generalization of dualizability, is important in this paper.  

\begin{example}\label{ex:bimod}  Let $\VV$ be a monoidal category.  The bicategory $\ALG_{\VV}$ of bimodules in $\VV$ has as objects all algebras in $\VV$, while a $1$-morphism from $A$ to $B$ is an $A$-$B$-bimodule ${}_{A}X_{B}$, and a $2$-morphism from ${}_{A}X_{B}$ to ${}_{A}Y_{B}$ is a morphism of $A$-$B$-bimodules.  If $X:A\to B$ and $Y:B\to C$ are $1$-morphisms, then their composite is defined to be $X\tensor_{B}Y: A \to C$. Note that for any object $A$, the identity morphism on $A$ is $A$ itself, seen as an $A$-$A$-bimodule. 

If there is an adjunction
$$\adjunction {A}{B}{X}{Y};\qquad u\colon A \to X \tensor _{B} Y, \quad e\colon Y\tensor _{A}X \to B,$$
then we say that $X$ is \emph{right dualizable} and call $Y$ a \emph{right dual} of $X$ and $X$ a \emph{left dual} of $Y$.  Notice that 
 the composites
$$X \xrightarrow {u\otimes 1} X\otimes_{B} Y\otimes_{A} X \xrightarrow {1\otimes e} X\quad\text{and}\quad Y \xrightarrow{1\otimes u} Y\otimes_{A} X \otimes_{B} Y \xrightarrow {e\otimes 1} Y$$
are both identities.

To emphasize the difference between this definition and that of section \ref{sec:dual-bimod}, in which homotopy of bimodules is taken into account, we sometimes say that a bimodule is \emph{strictly dualizable} if it is dualizable in the sense of this example.
\end{example}

Motivated by the fact that the usual notion of dual is a special case of the notion of adjoint in a bicategory, we introduce the following notation.

\begin{notation}  Let $\mathsf C$ be a bicategory, and let $l\dashv r$ be an adjunction in $\mathsf C$.  We then write
$$l^{\vee}:=r.$$
\end{notation}

The next lemma is well known to many category theorists and homotopy theorists; one proof can be found in \cite[\S III.1, Proposition 1.3]{lewis-may-etal}.  

\begin{lemma}\label{lem:char-dual} Let $\VV$ be a closed monoidal category.  Let $A$ and $B$ be algebras in $\VV$.

The following are equivalent for any $A$-$B$-bimodule $X$.
\begin{enumerate}
\item $X$ is right dualizable.
\item For all right $B$-modules $N$, the map $\ell_{N}: N\tensor _{B}\Map_{B}(X,B)\to \Map_{B}(X,N)$ that is adjoint to 
$N\tensor _{B}\Map_{B}(X,B)\tensor _{A}X\xrightarrow {1\otimes \mathrm{ev}} N$ is a isomorphism.
\end{enumerate}
When $X$ is dualizable, the $B$-$A$-bimodule $\Map_{B}(X,B)$ is a right dual to $X$.
\end{lemma}

\begin{proof}  Since the proof can be found elsewhere, we simply remind the reader how to prove the last statement.
The right adjoint of a functor is uniquely determined up to isomorphism. The functor $-\tensor_A X\colon \VV_A\rightarrow \VV_B$ always has right adjoint $\Map_B(X,-)$. Hence, $Y$ is right dual to $X$ if and only if there is a natural isomorphism of $A$-modules
$$N\tensor_B Y \cong \Map_B(X,N).$$
Plugging in $N = B$, we see that $Y$ must be isomorphic to $\Map_B(X,B)$.
\end{proof}

\bibliographystyle{amsplain}
\bibliography{vstructures}
\end{document}